\documentclass[english,a4paper,11pt]{article}
\usepackage[margin=3cm]{geometry}
\usepackage[latin1]{inputenc}
\usepackage{babel}
\usepackage{bm}
\usepackage{amsmath,amsthm}
\usepackage{booktabs}
\usepackage[final]{graphicx}
\DeclareGraphicsExtensions{.jpg,.jpeg,.pdf,.png,.mps}
\usepackage{amsfonts}
\usepackage{footnote}
\usepackage{multicol}
\usepackage{booktabs}
\usepackage[flushleft]{threeparttable}
\usepackage[font=small,labelfont=bf]{caption}
\usepackage[numbers]{natbib}
\usepackage{graphicx} 
\usepackage{subcaption} 
\usepackage{epsfig}
\usepackage{parskip}
\setcitestyle{square}

\usepackage{rotating}
\usepackage{color}
\usepackage{xcolor}
\usepackage{colortbl}
\usepackage[T1]{fontenc}
\usepackage{threeparttable}
\usepackage{lipsum}
\usepackage{url} 
\usepackage{hyperref}
\usepackage{lmodern}
\date{}
\hypersetup{hidelinks}

\theoremstyle{plain}
\newtheorem{theorem}{Theorem}[section]
\newtheorem{corollary}[theorem]{Corollary}
\newtheorem{lemma}[theorem]{Lemma}

\newtheorem{proposition}[theorem]{Proposition}

\theoremstyle{definition}

\newtheorem{definition}[theorem]{Definition}

\newtheorem{remark}[theorem]{Remark}

\usepackage{shuffle}

\title{Regularized Learning for Fractional Brownian Motion via Path Signatures}

\author{$\mathrm{Ali \ Mohaddes}^\mathrm{1},  
	\  \mathrm{Francesco \ Iafrate}^\mathrm{1},
	\  \mathrm{Johannes \ Lederer}^\mathrm{1}$\\  
	$^\mathrm{1}$\small{\emph{University of Hamburg, Hamburg, Germany}}
}

\begin{document}
	\maketitle

	\begin{abstract}
		Fractional Brownian motion (fBm) extends classical Brownian motion by introducing dependence between increments, governed by the Hurst parameter 
		$H\in (0,1)$. Unlike traditional Brownian motion, the increments of an fBm are not independent. Paths generated by fractional Brownian motions can exhibit significant irregularity, particularly when the Hurst parameter is small.
		As a result, classical regression methods may not perform effectively. Signatures, defined as iterated path integrals of continuous and discrete-time processes, offer a universal nonlinearity property that simplifies the challenge of feature selection in time series data analysis by effectively linearizing it. Consequently, we employ Lasso regression techniques for regularization when handling irregular data. 
		To evaluate the performance of signature Lasso on fractional Brownian motion (fBM), we study its consistency when the Hurst parameter $ H \ne \frac{1}{2} $. This involves deriving bounds on the first and second moments of the signature. For the case $ H > \frac{1}{2} $, we use the signature defined in the Young sense, while for $ H < \frac{1}{2} $, we use the Stratonovich interpretation. 
		 Simulation results indicate that signature Lasso can outperform traditional regression methods for synthetic data as well as for real-world datasets.
		\hspace{0.9cm}
	\end{abstract}

	\section{Introduction}\label{sec1}
	
	Fractional Brownian motion generalizes standard Brownian motion by  considering processes with dependent increments. The trajectories produced by fractional Brownian motions often exhibit significant irregularity. This irregularity can challenge the effectiveness of conventional regression techniques. In order to overcome this, signatures, defined as iterated path integrals, offer a robust representation of nonlinear functionals that streamlines feature selection in time series analysis by linearizing the model. Signatures are especially advantageous for examining irregular trajectories. In this paper, we explore regression techniques tailored for data generated by a fractional Brownian motion, utilizing the concept of path signatures.
	
	\textbf{Fractional Brownian Motions.}
	A fractional Brownian motion (fBm) $ B = \{B (t), t \in \mathbb{R}\} $  is a zero mean Gaussian process, defined on a complete probability space $ (\Omega, \mathcal{F}, P) $.
	A $d$-dimensional fractional Brownian motion (fBm), denotes by $\boldsymbol{B}_t = (B_t^1, \ldots, B_t^d)$, with Hurst parameter $H \in (0,1)$. This process is a zero-mean Gaussian process with independent components, each characterized by the covariance function:
	\[
	\mathbb{E}[B_t^i B_s^i] = \frac{1}{2} \left( t^{2H} + s^{2H} - |t - s|^{2H} \right), \quad s, t \in \mathbb{R}_+.
	\]
	
	For $H = \frac{1}{2}$, this corresponds to standard Brownian motion. For any $H \in (0,1)$, the variance of the increments is given by:
	\[
	\mathbb{E}\left[ (B_t^i - B_s^i)^2 \right] = (t - s)^{2H}, \quad 0\le s<t\le T, \; i = 1, \ldots, d.
	\]

	For $ H < \frac{1}{2} $, each component $ B_t^i $ of $\boldsymbol{B}_t  $ can be written as
	\begin{equation}\label{fBmRepresentation}
		B_t^i = \int_{\mathbb{R}} K(t, u)\, dW_u^i, \quad t \geq 0,
	\end{equation}
	where $\boldsymbol{W}_t  = (W^1_t, \ldots, W^d_t) $ is a $ d $-dimensional Wiener process, and the Volterra-type kernel $ K(t,u) $ is defined on $ \mathbb{R}_+ \times \mathbb{R}_+ $ by
	\begin{align}
		K(t,u) = c_H &\left[ \Big( \frac{u}{t} \Big)^{1/2 - H} (t - u)^{H - 1/2} \right. \notag \\
		&\left. + \big( \frac{1}{2} - H \big) u^{1/2 - H} \int_u^t v^{H - 3/2}(v - u)^{H - 1/2} dv \right] \mathrm{1}_{\{0 < u < t\}},
	\end{align}
 where the normalizing constant $ c_H $ is is a constant depending on $ H $
	\[
	c_H = \bigg( \frac{2H \Gamma\big( \frac{3-2H}{2} \big)}{\Gamma(2-2H) \Gamma(H + \frac{1}{2})} \bigg)^{\frac{1}{2}},
	\]
	where $ \Gamma $ is the Gamma function \citep{nualart2006}. 
	In general, the model can capture three types of processes. When $H > \frac{1}{2}$, the process increments are positively correlated, indicating persistence or long memory. This means consecutive increments tend to have the same sign, and a random step in one direction is likely to be followed by another step in the same direction. When $0 < H < \frac{1}{2}$, the increments are negatively correlated. In this case, a random step in one direction is typically followed by a step in the opposite direction. 
	When $H = \frac{1}{2}$, the process follows standard Brownian motion, where the increments are independent. Fractional Brownian motions lack regularity. The sample paths of fBm are almost surely $\gamma$-H\"older continuous for all $\gamma < H$, which underlines its role as a canonical example in rough path theory. Within rough path theory, the signature serves as a core construct that captures the essential information of a path.

	\textbf{Signature of a Path.} The signature transform, introduced by Chen \citep{Chen1957}
	and later expanded by Lyons \citep{lyons2007differential}, 
	has become a foundational tool in time series analysis. The signature of a time series is a vector of real-valued features capturing relevant information. 
	Consider a  $d$-dimensional continuous-time stochastic process $\boldsymbol{X}_t = (X_t^1, X_t^2, \dots, X_t^d)^\top \in \mathbb{R}^d$, defined over the time interval $0 \leq t \leq T$ on a filtered probability space $\left(\Omega, \mathcal{F}, \left\{\mathcal{F}_t\right\}, \mathbb{P} \right)$. The signature, or signature transform, of this process is defined as follows:
	
	\begin{definition}[Signature] For $k \geq 1$ and indices $i_1, \dots, i_k \in \{1, 2, \dots, d\}$, the $k$-th order signature component of the process $\boldsymbol{X}$ from time 0 to $t$ is given by \begin{equation} \label{equ: def_of_signature} S_{{i_1,\dots,i_k}}(\boldsymbol{X})_{t} = \int_{0<t_1<\dots<t_k<t} \mathrm{d} X_{t_1}^{i_1} \cdots \mathrm{d} X_{t_k}^{i_k}, \quad 0 \leq t \leq T. \end{equation} The $0$-th order signature component of $\boldsymbol{X}$ from time $0$ to $t$ is defined as $S_0(\boldsymbol{X})_{t}= 1$ for any $0 \leq t \leq T$. The signature of $X$ is the collection of all its signature components, and the signature of $\boldsymbol{X}$ with orders truncated to $K$ consists of all components with orders no greater than $K$.
	\end{definition}
	The $k$-th order signature component of $\boldsymbol{X}$, as defined in \eqref{equ: def_of_signature}, represents the $k$-fold iterated path integral along the indices $i_1, \dots, i_k$. For any given order $k$, there are $d^k$ possible index combinations $I = (i_1, \dots, i_k)$, so the total number of $k$-th order signature components is $d^k$. Signature components can be computed recursively as: 
	\begin{equation*} S_{I}(\boldsymbol{X})_{t} = S_{i_1,\dots,i_k}(\boldsymbol{X})_{t} = \int_{0<t_1<\dots<t_k<t} \mathrm{d} X_{t_1}^{i_1} \cdots \mathrm{d} X_{t_k}^{i_k} = \int_{0<s<t} S_{i_1,\dots,i_{k-1}}(\boldsymbol{X})_{s} \mathrm{d} X_{s}^{i_k}. 
	\end{equation*}
	
When  $ H > \frac{1}{2} $, the paths Holder continuity ($ \alpha < H $) allows to apply Young integration. Let $X$ be $\alpha$-Holder continuous, and let $Y$ be $\beta$-Holder continuous and let $\pi^n = \{0 = t_0^n < t_1^n < \cdots < t_{N_n}^n = T\}, \, n \geq 1$, be a sequence of partitions with vanishing mesh size. Young integration is defined as 
\begin{equation*}
	\int_0^T Y_s \, dX_s := \lim_{n \to \infty} \sum_{i=0}^{N_n - 1} Y_{u_i^n}(X_{t_{i+1}^n} - X_{t_i^n})
\end{equation*}

	Young integration converges for fBm with $ p < \frac{1}{H} $ when $ H > \frac{1}{2} $ \citep{young1936}. For $ H < \frac{1}{2} $, fBms trajectories are  more irregular paths and we cannot anymore apply Young integration.  Therefore we need to apply alternative integration definition. In this paper we apply Stratonovich integration, defined as
\[
\int_0^T Y_s \circ \mathrm{d}X_s = \lim_{n \to \infty} \sum_{i=0}^{N_n - 1} \frac{1}{2} \left( Y_{t_i^n} + Y_{t_{i+1}^n} \right) \left( X_{t_{i+1}^n} - X_{t_i^n} \right). 
\]
For simplicity in this paper we use $\int_0^T Y_s  \mathrm{d}X_s $ to denote Young or Stratonovich integration. Another important tools in theory of rough paths is \emph{shuffle product} operator.
The shuffle product operator, which is defined as a way of combining iterated integrals by interleaving their indices while preserving order within each integral, is a valuable tool that simplifies computations involving signatures.

\begin{definition}
Let $\mathrm{Shuffles}(k, m)$ denote the set of all $(k, m)$-shuffles. Consider two multi-indexes $I_k = (i_1, \ldots, i_k)$ and $J_m = (j_1, \ldots, j_m)$ with $i_1, \ldots, i_k, j_1, \ldots, j_m \in \{1, \ldots, d\}$. Define the multi-index
\[
(r_1, \ldots, r_k, r_{k+1}, \ldots, r_{k+m}) = (i_1, \ldots, i_k, j_1, \ldots, j_m).
\]
The shuffle product of $I$ and $J$, denoted $I \shuffle J$, is a finite multi-set of multi-indexes of length $k + m$ defined by
\[
I_k \shuffle J_m = \{ (r_{\tau(1)}, \ldots, r_{\tau(k+m)}) \mid \tau \in \mathrm{Shuffles}(k, m) \}.
\]

\end{definition}

	\begin{definition}[Shuffle Product for Signatures]
		For a path $\boldsymbol{X}$ and multi-indices $ I_p = (i_1, \ldots, i_p) $ and $ J_p = (j_1, \ldots, j_q) $, where $i_r , j_r \in \{1,\dots,d\}$ we have
		\[
		S_I(\boldsymbol{X})  S_J(\boldsymbol{X}) = \sum_{K \in {I_p \shuffle J_q}} S_K(\boldsymbol{X}),
		\]
		where $ {I_p \shuffle J_q} $ denotes shuffles preserving the order of $ \{i_1, \ldots, i_p\} $ and $ \{j_1, \ldots, j_q\} $. For example, $ S_{1}(\boldsymbol{X}) \cdot S_{2}(\boldsymbol{X}) = S_{1,2}(\boldsymbol{X}) + S_{2,1}(\boldsymbol{X}) $.   
	\end{definition}

One important advantage of employing Stratonovich integration is its compatiblity with the shuffle product.  Iterated Stratonovich integrals, under pointwise multiplication, form a shuffle algebra, and allowing us to leverage results from shuffle algebra theory \cite{gaines1995basis}. 

Statistical estimation for fBm has received a lot of attention. See e.g. \citep{berzin2007estimating} for the estimation of the Hurst parameter, or  \citep{xiao2011parameter} for parametric estimation in the case of an Ornstein-Uhlenbeck (OU) process driven by an fBm and \citep{mishura2008stochastic} for a general overview.
Our approach is different in the sense that we directly approach the problem of approximating functionals of the trajectories. This is possible by means of signature methods. 

	One of the signature's most notable theoretical properties is its universal nonlinearity: any continuous function of a time series, whether linear or nonlinear, can be approximated arbitrarily well using a linear combination of its signature components. Such  linearity property gives the signature an edge over neural-network-based nonlinear methods \citep{guo2023consistency}.
	Empirical evidence highlights its advantages, including the absence of complex neural architecture design requirements.
	The signature has demonstrated strong performance in various fields such as machine learning, finance, and medical prediction \citep{lyons2014feature,lyons2022signature,kormilitzin2017detecting}.
	Additionally, few studies explore its statistical aspects beyond probabilistic characteristics \citep{fermanian2022functional, guo2023consistency}.

	\section{Lasso regression for fractional Brownian motions}
	As stated above, the key theoretical property of the signature is its universal nonlinearity. The following theorem formally expresses this property.
	
	\begin{theorem}[Universal nonlinearity, {\citep[Theorem 2.12]{cuchiero2023signature}}]\label{th:UN}
		Let $\boldsymbol{X}_t$ be a continuous $\mathbb{R}^d$-valued and $\mathcal{S}$ be a compact subset of paths of the time-augmented process $\tilde{\boldsymbol{X}}_t = \begin{pmatrix}
			t , {\boldsymbol{X}}_t^\top 
		\end{pmatrix}^\top$ from time 0 to $T$. 
		Assume that $f: \mathcal{S} \to \mathbb{R}$ is a real-valued continuous function. Then, for any $\varepsilon>0$, there exists a linear functional $L: \mathbb{R}^\infty \to \mathbb{R}$ such that
		\begin{equation*}
			\sup_{s \in \mathcal{S}}\left| f (s) - L(\mathrm{Sig}(s)) \right| < \varepsilon,
		\end{equation*}
		where $\mathrm{Sig}(s)$ is the signature of $s$. 
	\end{theorem}
	This Theorem follows from the Stone-Weierstrass theorem. The classical Weierstrass approximation theorem asserts that any real-valued continuous function on a closed interval can be uniformly approximated by a polynomial. Therefore, linear forms on the signature can be considered analogous to polynomial functions for paths.

	\subsection{Signature Lasso regression}
	
	Let $\left(\boldsymbol{X}_1, y_1\right), \ldots, \left(\boldsymbol{X}_N, y_N\right)$, denote $N$ pairs of samples, where~$\boldsymbol{X}_n=\left\{\boldsymbol{X}_{n,t}, \,0\leq t \leq T \right\}$ represents the $n$-th path realization of $\boldsymbol{X}_t$ for $n = 1, 2, \ldots, N$. 
	Motivated by the universal nonlinearity, for a fixed order $K \geq 1$, assume that each pair $\left(\boldsymbol{X}_n, y_n\right)$ follows the linear regression model given below \citep{guo2023consistency}:
	\begin{equation}\label{SignatureModel}
		y_n=\beta_0+\sum_{i_1=1}^d \beta_{i_1} S_{i_1}\left(\tilde{\boldsymbol{X}}_n\right)_T+\sum_{i_1, i_2=1}^d \beta_{i_1, i_2} S_{i_1, i_2}\left(\tilde{\boldsymbol{X}}_n\right)_T+\cdots+\sum_{i_1, \ldots, i_K=1}^d \beta_{i_1, \ldots, i_K} S_{i_1, \ldots, i_K}\left(\tilde{\boldsymbol{X}}_n\right)_T+\varepsilon_n,   
	\end{equation}
	where $ \left\{\varepsilon_n\right\}_{n=1}^N $ are independent and identically distributed random errors, assumed to follow a zero mean normal distribution. The term $ {S}\left(\tilde{\boldsymbol{X}}_n\right) $ denotes the signature of the augmented process $\tilde{\boldsymbol{X}}$. The total number of predictors, which are the signature components across different orders, is  
	$p_{d, K} = (d^{K+1} - 1)/(d - 1)$ if $d>1$, 
	including the $0$th-order signature component $ S_0(\boldsymbol{X})_T = 1 $, with its coefficient denoted by $ \beta_0 $. 
	
	Given that the number of parameters grows exponentially with the truncation level, we are in a high dimensional regime with $p_{d, K} \gg N$, potentially. 
	We then pursue a sparse representation using Lasso regression \citep{Tibshirani1996}, a widely used method for learning sparse models. This paper aims to investigate the statistical properties of signature representations in Lasso regression for fractional Brownian motions.
	With a tuning parameter $ \lambda > 0 $ and $ N $ data points, the Lasso estimator selects the true predictors through
	\begin{align}\label{SignatureLasso}
		\hat{\boldsymbol{\beta}}^N(\lambda)=\arg \min _{{\boldsymbol{\beta}}}\Bigg[\sum _ { n = 1 } ^ { N } \bigg(y_n-{\beta}_0 & -\sum_{i_1=1}^d {\beta}_{i_1} {S}_{i_1}\left(\tilde{\boldsymbol{X}}_n\right)_T-\sum_{i_1, i_2=1}^d {\beta}_{i_1, i_2} {S}_{i_1, i_2}\left(\tilde{\boldsymbol{X}}_n\right)_T-\ldots \\
		& -\sum_{i_1, \ldots, i_K=1}^d {\beta}_{i_1, \ldots, i_K} {S}_{i_1, \ldots, i_K}\left(\tilde{\boldsymbol{X}}_n\right)_T\bigg)^2+\lambda\|{\boldsymbol{\beta}}\|_1 \Bigg], \notag
	\end{align}
	where $ \hat{\boldsymbol{\beta}} $ is the vector comprising all coefficients $ \hat{\beta}_{i_1, \ldots, i_k} $. 
	
	Our objective is to examine the consistency of feature selection for fractional Brownian motion processes using the signature method with the Lasso estimator described in equation (\ref{SignatureLasso}). Generally, consistency refers to the convergence of the Lasso estimator to the true coefficient values as the sample size grows. To investigate the consistency, we need to analyze the covariance between signature terms \cite{guo2023consistency}. In the following section, we bound the first and second moment of signature terms associated with fractional Brownian motion.

	\section{Upper bound for covariance of signature terms}
	
	In this section, we determine an upper bound for the covariance between signature terms. Since our approach varies depending on whether $H>1/2$ 
	or $H<1/2 $, we present the methodology in the subsequent subsections.
	
	\subsection{Covariance when $H>1/2$}

	If $p+q$ be an odd number then according to \cite{baudoin2007operators} this expectation of the above integral is equal to zero if $p+q = 2k=n$ then we have the below theorem.
	
	\begin{theorem}[Upper bound for the second moment of fBm signature]\label{UpperBoundYoung}
		Consider a fractional Brownian motion with the Hurst parameter  $ H > \frac{1}{2} $. 
		Let $S_{I_p}(\boldsymbol{B}_{t})_{st}$ with $I_p = \{i_1,\dots,i_p \}$ and $S_{I_p}(\boldsymbol{B}_{t})_{st}$ with $J_q = \{j_1,\dots,j_q \}$ denote $p$th and $q$th order integrals of a $d$-dimensional fractional Brownian motion process and let  $p+q=2k=n$. Let  $0 \le s <  t \le T$,  
		 then
		\[
		\mathbb{E}[S_{I_p}(\boldsymbol{B}_{t})_{st}S_{J_q}(\boldsymbol{B}_{t})_{st}] \leq \frac{2^{2k}}{k!} (t-s)^{2kH}.
		\]

	\end{theorem}

	\subsection{Covariance when $H<1/2$}
	
		When $H < \frac{1}{2}$, we consider a family of rough paths associated with a $d$-dimensional fractional Brownian motion (fBm), based on the representation given in Equation~(\ref{fBmRepresentation}). For $n = 1$, the first level is given by
	\[
	S_{I_1} (B_t)_{st}  = B_t^i - B_s^i,
	\]
	and for $2 \le n \le \lfloor 1/H \rfloor$, it is extended to higher-order iterated integrals constructed over $\boldsymbol{B}_t$ as 
	\begin{equation*}
		S_{I_n}(\boldsymbol{B}_t)_{st}  = \sum_{j=1}^{n} (-1)^{j-1} \int_{A_j^n} \prod_{l=1}^{j-1} K(s, u_l)[K(t, u_j) - K(s, u_j)] \prod_{l=j+1}^{n} K(t, u_l) \, dW_{u_1}^{i_1} \cdots dW_{u_n}^{i_n},
	\end{equation*}
	where integration is understood in the Stratonovich sense and $A_j^n$ is the subset of $[0,t]^n$ defined by
	\[
	A_j^n = \{(u_1, \ldots, u_n) \in [0,t]^n \colon u_j = \min(u_1, \ldots, u_n),\ u_1 > \cdots > u_{j-1} \text{ and } u_{j+1} < \cdots < u_n\}.
	\]
	Sometimes for simplicity we write above equation as
	
	\begin{equation}\label{RoughPathTerms}
		S_{I_n}(\boldsymbol{B}_t)_{st} = \sum_{j=1}^{n} \hat{S}^j_{I_n}({\boldsymbol{B}_t})_{st}.
	\end{equation}
	
Then we will have the following bounds for the signature moments.
	
	\begin{lemma}[Upper bound for first moment of fBm signature]\label{SignatureFirstMoment}Consider a fractional Brownian motion with the Hurst parameter  $ H < \frac{1}{2} $. 
		Let $S_{I_n}(\boldsymbol{B}_{st})$ with $I_n = \{i_1,\dots,i_n \}$ 
		For fBm when $n$ is odd then 
		\[
		\mathbb{E}[S_{I_n}(\boldsymbol{B}_t)_{st}] = 0.
		\]
		Otherwise for $n=2k$
	\[
\mathbb{E}[S_{I_n}(\boldsymbol{B}_{t})_{st}]  \le \frac{{\beta_{{k,H}}}}{k!H^k} (t - s)^{nH}
\]
where
\begin{equation*}\label{BoundIntA}
	{\beta_{{k,H}}}= \frac{\pi (\frac{1}{2}-H)}{\cos(\pi H)} + \frac{1}{1-2kH} 
\end{equation*}
	\end{lemma}
	
	\begin{proof}
		Proof is deferred to Appendix.
	\end{proof}

	\begin{proposition}[Upper bound for second moment of fBm signature]\label{SignatureExpectation}
		For $n \leq \left\lfloor \frac{1}{H} \right\rfloor$, 
		let $S_{I_n}(\boldsymbol{B}_{st})$ with $I_n = \{i_1,\dots,i_n \}$ denote $n$th order integrals of a $d$-dimensional  fractional Brownian motion process. Then for $0 \le s <  t \le T$, we have
		\begin{equation}
	\mathbb{E}[|S_{I_n}(\boldsymbol{B}_{t})_{st}|^2] \leq \frac{2^{4n}}{n!H^n}   (n^2+\frac{2^{2n}n^5 \beta_{n,H}}{H^n}) \times (t - s)^{2nH }
		\end{equation} 
where $\beta_{n,H} = \frac{\pi (\frac{1}{2}-H)}{\cos(\pi H)} + \frac{1}{1-2nH} $.

	\end{proposition}
	
	\begin{proof}
		Proof is deferred to Appendix. 
	\end{proof}

		\begin{theorem}[Upper bound for covariance of fBm signature]
		
		For $n \leq \left\lfloor \frac{1}{H} \right\rfloor$, 
		let $S_{I_p}(\boldsymbol{B}_{t})_{st}$ with any tuple $I_p = \{i_1,\dots,i_p \}$ of elements of $\{1,\dots,d\}$ and $S_{I_p}(\boldsymbol{B}_{t})_{st}$ with any tuple $J_q = \{j_1,\dots,j_q \}$ of elements of $\{1,\dots,d\}$denote $p$th and $q$th order integrals of a $d$-dimensional fractional Brownian motion Process when $p+q=2k = n$. Then for  $0 \le s <  t \le T$, we have
		\[
		\mathbb{E}[S_{I_p}(\boldsymbol{B}_{t})_{st} S_{J_q}(\boldsymbol{B}_{t})_{st}] \leq  \frac{2^n \beta_{k,H}}{k! H^k} (t - s)^{2kH}
		\] 
		where $\beta_{k,H} = \frac{\pi (\frac{1}{2}-H)}{\cos(\pi H)} + \frac{1}{1-2kH} $.
		
	\end{theorem}
	\begin{proof}
		From shuffle product property we know that 
		\[
		\mathbb{E}[S_{I_p}(\boldsymbol{B}_{t})_{st}S_{J_q}(\boldsymbol{B}_{t})_{st}] = \sum_{K_n \in 	{I_p \shuffle J_q} } \mathbb{E}[S_{K_{n}}(\boldsymbol{B}_{t})_{st}]
		\]

		We know that $|{I_p \shuffle J_q}| = \frac{n!}{p!q!};  p+q = 2k = n$. Therefore applying Lemma \ref{SignatureFirstMoment} we have:
		\begin{align*}
			\mathbb{E}[S_{I_p}(\boldsymbol{B}_{t})_{st}S_{J_q}(\boldsymbol{B}_{t})_{st}] &\le \binom{n}{p}  \times\frac{{\beta_{{k,H}}}}{k!H^k} (t - s)^{nH} \\
			& \le \frac{2^n \beta_{k,H}}{k! H^k} (t - s)^{2kH}
		\end{align*}
		where  $\beta_{k,H} = \frac{\pi (\frac{1}{2}-H)}{\cos(\pi H)} + \frac{1}{1-2kH} $.
		
	\end{proof}

	\section{Simulation results}
	
	This section examines signature regression on simulated and real-world data. We first focus on applications in financial time series analysis. 
	We also study the \emph{UCI Air Quality} dataset, evaluating both methods using mean squared error for in-sample and test data.
	
	\textbf{Fractional Brownian Motion Analysis:}
	In this section we show results for simulated data in a financial setting. We generated sample paths $\boldsymbol{X}_i$ of a fractional Brownian motion representing the price of a stock. The outputs $y_i$ represent option prices produced by  the following payoff functions, with time to maturity $T =1$: \\
	(I) Call option $(d=1): \max \left(X_T^1-1.2,0\right)$; \\
	(II) Asian option $(d=1): \max \left(\operatorname{mean}_{0 \leq t \leq T}\left(X_t^1\right)-1.2,0\right)$; \\
	(III) Rainbow option I $(d=2): \max \left(X_T^1-X_T^2, 0\right)$; \\
	(IV) Rainbow option II $(d=2)$ : $\max \left(\max \left(X_T^1, X_T^2\right)-1.2,0\right)$;
	
	We then applied Lasso regression and signature regression to analyze the outputs.
	The signature computation was truncated at level $K=3$. The results were compared across different Hurst parameter values. The results are illustrated in Figures \ref{fig:CallOptions}, \ref{fig:PutOption}, \ref{fig:RainbowOptionI}, and \ref{fig:RainbowOptionII}. Where we graphed the test error against the Hurst parameter.  As anticipated, the signature method demonstrated greater efficiency than conventional regression methods as the irregularity of the data increased.
	The performance of both Lasso and Signature Lasso improved with an increase in the Hurst parameter. This is because higher Hurst parameters correspond to smoother signals, making it easier for regression methods to approximate the underlying function. 
	We note that for Asian and Rainbow II options the dependency on $H$ is stronger. This might be due to the fact that such options depend on the full path of the process. Call and Rainbow I options instead, are a function of $X_T$ only, i.e. the process at the last time instant.

	\begin{figure}[ht]
		\centering
		\begin{subfigure}[b]{0.37\textwidth} 
			\centering
			\includegraphics[width=\textwidth]{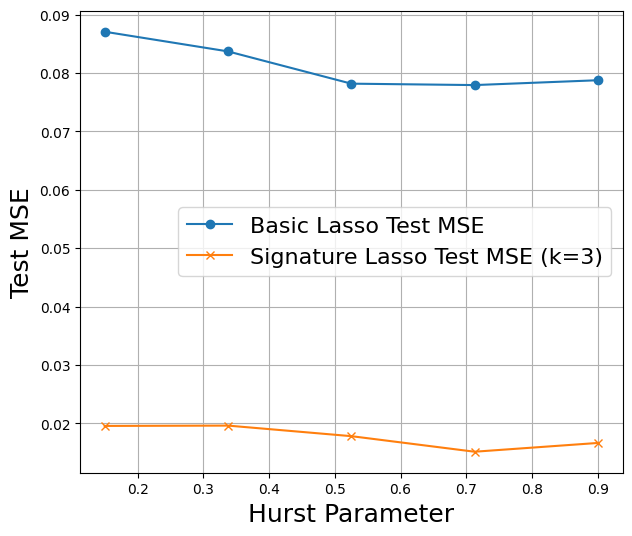}
			\caption{Lasso regression comparing with signature regression for Call options.}
			\label{fig:CallOptions}
		\end{subfigure}
		\hfill
		\begin{subfigure}[b]{0.37\textwidth}
			\centering
			\includegraphics[width=\textwidth]{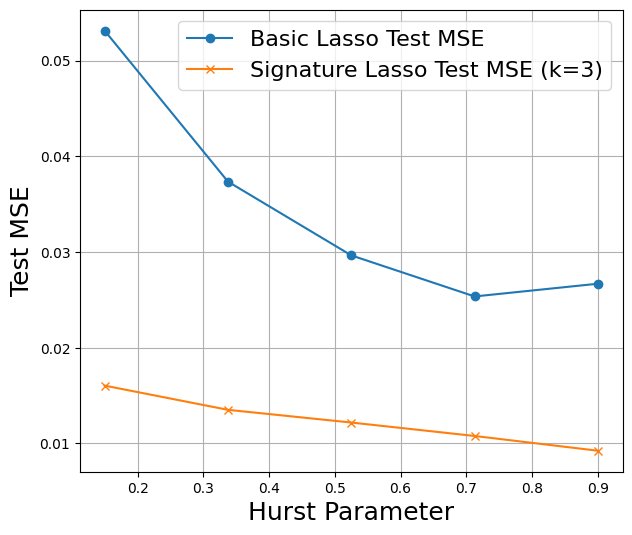}
			\caption{Lasso regression comparing with signature regression for Asian options.}
			\label{fig:PutOption}
		\end{subfigure}
		
		\vspace{1em} 
		
		\begin{subfigure}[b]{0.37\textwidth}
			\centering
			\includegraphics[width=\textwidth]{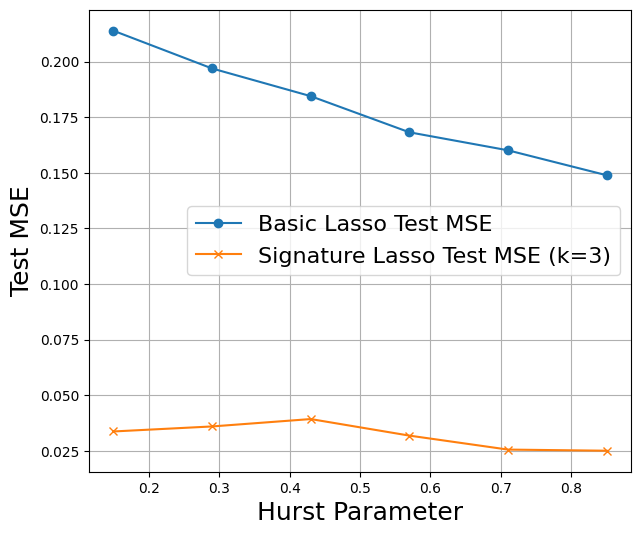}
			\caption{Lasso regression comparing with signature regression for Rainbow Option I.}
			\label{fig:RainbowOptionI}
		\end{subfigure}
		\hfill
		\begin{subfigure}[b]{0.37\textwidth}
			\centering
			\includegraphics[width=\textwidth]{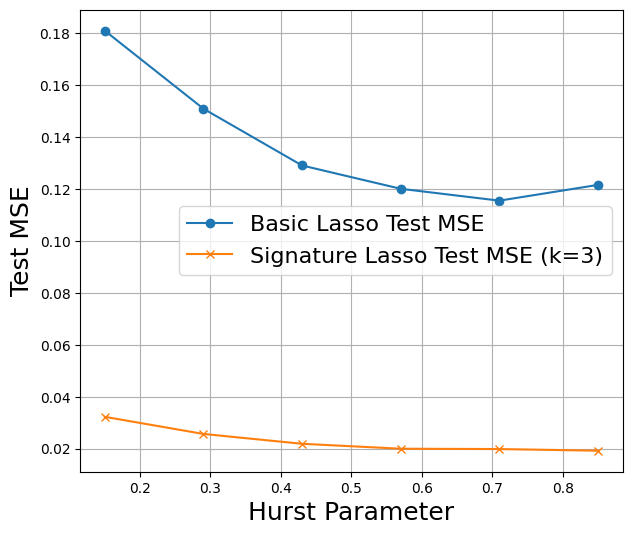}
			\caption{Lasso regression comparing with signature regression for Rainbow Option II.}
			\label{fig:RainbowOptionII}
		\end{subfigure}
		
		\caption{Comparison of Lasso and signature regression across various option functions shows that signature regression consistently outperforms Lasso.}
		\label{fig:main}
	\end{figure}

	\textbf{Air Quality Data.}
	We now move to the analysis of real world data. \emph{UCI Air Quality} dataset \citep{air_quality_360,fermanian2022functional}  contains hourly averages from five metal oxide chemical sensors recorded over a year in a polluted area in Italy. Our study focuses on the concentration of nitrogen dioxide (NO2), aiming to predict its ground truth value one hour into the future, based on the data from the previous 7 days.
	In order to create a multivariate setting, we included temperature and relative humidity as features in $ X $, representing a path in three dimensions ($ d = 3 $). 
	
	Specifically, we take $p=168$ previous time steps as inputs to our regression model, i.e, our dataset is 
	$\{(\boldsymbol{X}_{n}, y_n) , n =1, \ldots N\}$,
	where $\boldsymbol{X}_n = (X^i_{n-k},  k=1, \ldots, 168, i=1, 2, 3)$. 
	We applied Lasso regression and signature regression at two distinct levels $K = 2,3$, assessing the mean squared error for in-sample and test data. 
	In Figure \ref{fig:AirQuality} we show error curves as a function of the number of learning iterations during model training, by using a fixed tuning parameter. 
	For better visualization, 250 test time steps are selected, and the corresponding predictions are shown in \autoref{fig:pred} along with the MSEs. In order to determine the optimal tuning parameter, we employed a time-series cross-validation technique, by splitting the data in subsequent blocks of 900 training points followed by 100 test points. Simulation outcomes reveal that signature regression outperforms conventional regression methods.
	
	\begin{figure}[h]
		\centering
		\includegraphics[width=0.56\textwidth]{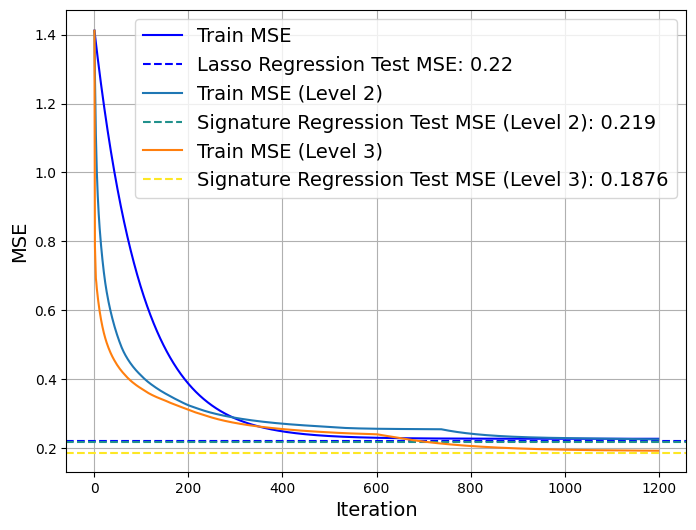}
		\caption{Lasso regression comparing to signature regression for Air Quality data.}
		\label{fig:AirQuality}
	\end{figure}
	
	\begin{figure}[h]
		\centering
		\includegraphics[width=1.1\textwidth]{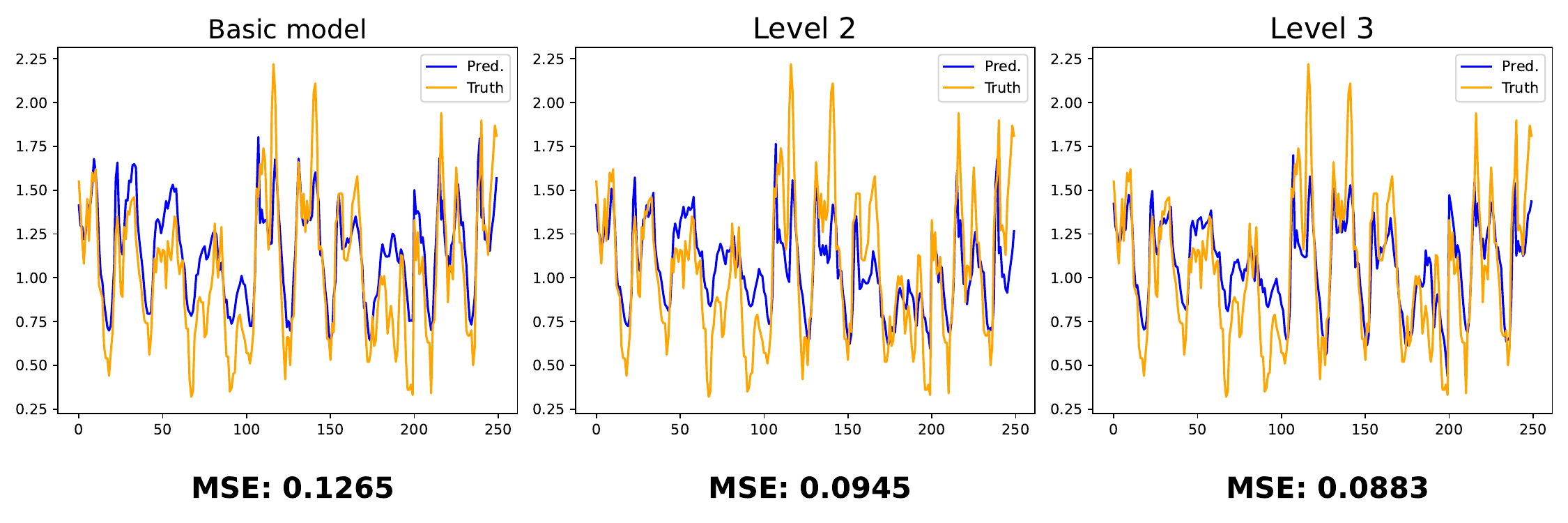}
		\caption{Predicted and true values for air quality data. From the left: standard Lasso regression, signature regression levels 2 and 3, respectively.}
		\label{fig:pred}
	\end{figure}

	\section{Conclusion}
	
This study demonstrates that signature-based Lasso regression is a powerful tool for analyzing time series data driven by fractional Brownian motion, particularly in regimes where classical methods struggle due to irregularity in paths. By leveraging the universal nonlinearity of path signatures and appropriately adapting to different interpretations based on the Hurst parameter, we establish theoretical consistency and practical performance gains. Simulation results further support the robustness and effectiveness of signature Lasso across both synthetic and real-world datasets, highlighting its promise for broader applications in irregular time series analysis.

	\section{Acknowledgment}
	
	This research was partially funded by grants 543964668 (SPP2298) and 520388526 (TRR391) by the Deutsche Forschungsgemeinschaft (DFG, German Research Foundation).

	\bibliographystyle{plainnat}

	\bibliography{references}
	
	\newpage
	
	\section{Appendix: Proofs}

	In order to provide proofs we need to review some notations.

\subsection{Notations}
	
	Before proceeding with the details, it is important to highlight and review some concepts and notations that are useful in our computations. 
	We denote the simplex over $[0,T]$ as
	\[
	\Delta^k_T = \Delta^k[0,T] = \{(u_1, \ldots, u_k) \in [0,t]^k,\, 0 < u_1 < \cdots < u_k \leq T\}
	\]

	When $H < \frac{1}{2}$,  the first level is given by
$S_{I_1} (B_t)_{st}  = B_t^i - B_s^i$
	and for $2 \le n \le \lfloor 1/H \rfloor$, it is extended to higher-order iterated integrals constructed over $\boldsymbol{B}_t$ as 
	\begin{equation*}
	S_{I_n}(\boldsymbol{B}_t)_{st}  = \sum_{j=1}^{n} (-1)^{j-1} \int_{A_j^n} \prod_{l=1}^{j-1} K(s, u_l)[K(t, u_j) - K(s, u_j)] \prod_{l=j+1}^{n} K(t, u_l) \, dW_{u_1}^{i_1} \cdots dW_{u_n}^{i_n},
	\end{equation*}
	where integration is understood in the Stratonovich sense and $A_j^n$ is the subset of $[0,t]^n$ defined by
	\[
	A_j^n = \{(u_1, \ldots, u_n) \in [0,t]^n \colon u_j = \min(u_1, \ldots, u_n),\ u_1 > \cdots > u_{j-1} \text{ and } u_{j+1} < \cdots < u_n\}.
	\]
	Sometimes for simplicity we write above equation as
	
\begin{equation}\label{RoughPathTerms}
    	S_{I_n}(\boldsymbol{B}_t)_{st} = \sum_{j=1}^{n} \hat{S}^j_{I_n}({\boldsymbol{B}_t})_{st}.
\end{equation}
We need to calculate integrals of below form
\begin{equation}
	Q_{st} = \int_{0 < u_1 < \cdots < u_n < t} \delta K_{st} \prod_{i=2}^{n} K_{\tau_i} \, dW.
	\tag{30}
\end{equation}

Notice that in the expression above, we made use of the notation introduced at the beginning of the current section, and for $i = 1, \ldots, n$, we assume $\tau_i = s$ or $t$. 
Let us further decompose $Q$ into $Q = Q^1 + Q^2$, where
\begin{equation}
	\begin{aligned}
		Q^1_{st} &= \int_{s < u_1 < \cdots < u_n < t} K_t^{\otimes n} \, dW, \\
		Q^2_{st} &= \int_{0 < u_1 < \cdots < u_n < t, \, u_1 < s} \delta K_{st} \prod_{i=2}^{n} K_{\tau_i} \, dW,
	\end{aligned}
\end{equation}
where $\delta K_{st} = K(t, u_j) - K(s, u_j)$ and $	\prod_{j=1}^{n} K(t, u_j) = K_t^{\otimes n}.$ The following proposition offers a method to express iterated Stratonovich integrals in terms of iterated Ito integrals, allowing us to take advantage of useful properties of Ito integration, such as the Ito isometry.

	\begin{proposition}[Proposition 2.6 from \cite{nualart2011construction}]\label{StratonovichIto}
		Let $\boldsymbol{Y}_t = (Y^1_t, \dots, Y^n_t)$ be an $n$-dimensional martingale of Gaussian type, defined on an interval $[s,t]$, of the form
		\begin{align*}
			Y_u(j) = \int_s^u \psi_v(j) \, dW_v^{i_j}
		\end{align*}
		for a family of $L^2([s,t])$ functions $(\psi(1), \dots, \psi(n))$, a set of indices $(i_1, \dots, i_n)$ belonging to $\{1,\dots,d\}^n$, and where we recall that $(W^1, \dots, W^d)$ is a $d$-dimensional Wiener process. Then the following decomposition holds true:
		\begin{align}\label{StratonovichItoFormula}
			\int_{s \leq u_1 < \cdots < u_n \leq t} 
			dY_{u_1}^{i_1} \cdots dY_{u_n}^{i_n}
			= \sum_{k = \lfloor n/2 \rfloor}^{n} \frac{1}{2^{n-k}} 
			\sum_{\nu \in D_n^k} J_{st}(\nu).
		\end{align}
		
		In the above formula, the sets $D_n^k$ are subsets of $\{1,2\}^k$ given by
		\begin{align*}
			D_n^k = \bigg\{ \nu = (n_1, \dots, n_k); 
			\sum_{j=1}^k n_j = n \bigg\},
		\end{align*}
		and the Ito-type multiple integrals $J_{st}(\nu)$ are defined as follows:
		\begin{align*}
			J_{st}(\nu) = \int_{s \leq u_1 < \cdots < u_k \leq t} 
			\partial Z_{u_1}^1 \cdots \partial Z_{u_k}^k,
		\end{align*}
		where, setting $\sum_{l=1}^{j} n_l = m(j)$, we have
		\begin{align*}
			Z^j = Y^{i_{m(j)}} \quad \text{if } n_j = 1,
		\end{align*}
		and
		\begin{align*}
			Z_u^j = \left( \int_s^u \psi_v(m(j) - 1) \psi_v(m(j)) \, dv \right) 
			\mathbf{1}_{(i_{m(j)-1} = i_{m(j)})} 
			\quad \text{if } n_j = 2.
		\end{align*}
	\end{proposition}

Concerning $Q_{st}^2$, one can write 
\[
Q_{st}^2 = \sum_{j=1}^{n} B_{st}^j 
\]
where
\[
B_{st}^j = \int_{0 < u_1 < \cdots < u_j < s < u_{j+1} < \cdots < u_n < t} \delta K_{st} \prod_{i=2}^{n} K_{\tau_i} \, dW.
\]

for which we can write below decomposition to two independent terms  $B_{st}^j = C_{st}^j D_{st}^j$

\[
C_{st}^j = \int_{0 < u_1 < \cdots < u_j < s} \delta K_{st} \prod_{i=2}^{j} K_{\tau_i} \, dW
\]
\[
D_{st}^j = \int_{s < u_{j+1} < \cdots < u_n < t} K_t^{\otimes (n-j)} \, dW.
\]
Moreover, applying Proposition \ref{StratonovichIto} we have 
\[
C_{st}^{j} = \sum_{k=\lfloor j/2 \rfloor}^{j} \frac{1}{2^{j-k}} \sum_{\nu \in D_j^k} J_{0s}(\nu).
\]

\subsection{Appendix: Proofs}

\subsubsection*{Proof of Theorem \ref{UpperBoundYoung}}

In this section we focus on upper bounds when $H>1/2$. 

	\begin{proof}
	
	In order to prove this theorem we need below lemma from \cite{baudoin2007operators} and Remark \ref{SimplexCubeSymmetricFunc}. Let
	 $ I = (i_1, \ldots, i_k) \in \{1, \ldots, d\}^k $ is a word with length $ k $, then we use below notation
	\[
	\int_{\Delta^k_{[0,t]}} dB^{I_k} = \int_{0 < t_1 < \cdots < t_k \leq t} dB^{i_1}_{t_1} \cdots dB^{i_k}_{t_k}.
	\]
	
	\begin{lemma}[Lemma 10 of \cite{baudoin2007operators}]\label{GaussianVectorExpectation}
		Let $ G = (G^1, \dots, G^{2k}) $ be a centered Gaussian vector. We have
		\[
		\mathbb{E}[G^1 \dots G^{2k}] = \frac{1}{k!2^k} \sum_{\sigma \in \Theta_{2k}} \prod_{l=1}^k \mathbb{E}[G^{\sigma(2l)} G^{\sigma(2l-1)}],
		\]
		where $ \Theta_{2k} $ is the group of the permutations of the set $\{1, \dots, 2k\}$.
	\end{lemma}
	
	\begin{remark}\label{SimplexCubeSymmetricFunc}
		Consider a general symmetric function $ g(u_1,  u_2, \ldots, u_k) $, meaning for any permutation $ \sigma $ of the indices $ \{1, 2, \ldots, k\} $,
		\[
		g(u_1, u_2, \ldots, u_k) = g(u_{\sigma(2)}, \ldots, u_{\sigma(k)}).
		\]
		The cube $ [s, t]^{k} $ can be partitioned into $k!$ regions based on the ordering of $u_1, u_2, \ldots, u_k $, such as $s<  u_1 < u_2 < \cdots < u_k < t $, $ s< u_1 < u_3 < u_2 < u_4 < \cdots < u_k < t $, and so on. Since $ g $ is symmetric, the integral of $ g $ over each ordered region is the same. The integral over the simplex is the integral over the cube:
		\[
		\int_{[s, t]^{k}} g(u_1, u_2, \ldots, u_k) \, du_1 du_2 \cdots du_k = k! \times \int_{\Delta^k[s,t]} g(u_1, u_2, \ldots, u_k) \, du_1 du_2 \cdots du_k.
		\]
	\end{remark}
	
Let $p+q = 2k = n$ then from the shuffle product definition we have 
	\begin{equation*} \mathbb{E}\bigg[\int_{\Delta^{p}[0,1]} dB^{I_p} \int_{\Delta^{q}[0,1]} dB^{J_q} \bigg] = \mathbb{E}\bigg[\sum_{K_n \in {I_p \shuffle J_q}} \int_{\Delta^{2k}[0,1]} dB^{K_n}\bigg]
	\end{equation*}
	as a result we get
	
	\begin{equation*}
		\mathbb{E}\bigg[\sum_{K_{n}  \in {I_p \shuffle J_q}} \int_{\Delta^{2k}[0,1]} dB^{K_n}\bigg] =  \sum_{K_n \in {I_p \shuffle J_q}} \int_{\Delta^{2k}[0,1]}  \mathbb{E} \Big[ dB^{K_n} \Big].
	\end{equation*}
	Applying Lemma \ref{GaussianVectorExpectation}
	\[
	\mathbb{E} \Big[ dB^{K_n} \Big] = \frac{1}{k!2^k} \sum_{\sigma \in \Theta_{2k}} \prod_{l=1}^{k} \mathbb{E} \bigg[ dB_{t_{\sigma(\tau^{-1}(2l))}}^{i_{\sigma(\tau^{-1}(2l))}} dB_{t_{\sigma(\tau^{-1}(2l-1))}}^{i_{\sigma(\tau^{-1}(2l-1))}} \bigg],
	\]
	
		where $\tau$ is the shuffle permutation.  Utilizing the covariance function of the fractional Brownian motion, we consequently obtain
	\[
	\mathbb{E} \bigg[ dB_{t_{\sigma(\tau^{-1}(2l))}}^{i_{\sigma(\tau^{-1}(2l))}} dB_{t_{\sigma(\tau^{-1}(2l-1))}}^{i_{\sigma(\tau^{-1}(2l-1))}} \bigg] =  H(2H-1) \delta_{i_{\sigma(\tau ^{-1}(2l-1))}, i_{\sigma(\tau ^{-1}(2l))}} \times 
	|t_{\sigma(\tau ^{-1}(2l))} - t_{\sigma(\tau ^{-1}(2l-1))}|^{2H-2},
	\]

	In the second part, our goal is to find an upper bound for  
	\begin{equation}
		E \left[ \int_{\Delta^{p}[0,1]} dB^{I_p} \int_{\Delta^{q}[0,1]} dB^{J_q} \right] =
		\frac{H^{k}}{k! 2^k} (2H-1)^{k} \sum_{K_n \in {I_p \shuffle J_q}}  I,
	\end{equation}
	where 
	\begin{align*}
		I &=  \int_{\Delta^{2k}[0,1]} \sum_{\sigma \in \Theta_{2k}}
		\prod_{l=1}^{2k} \delta_{i_{\sigma(\tau ^{-1}(2l-1))}, i_{\sigma(\tau ^{-1}(2l))}} \times |t_{\sigma(\tau ^{-1}(2l))} - t_{\sigma(\tau ^{-1}(2l-1))}|^{2H-2} \, dt_1 \dots dt_{2k},
	\end{align*}
	Now we want to apply Remark \ref{SimplexCubeSymmetricFunc}. Theretofore, we need to check wether the symmetry condition in the remark holds. Summing over all permutations $\pi \in \Theta_{2k}$, the total integrand is:
	\[
	g(t_1, \dots, t_{2k}) = \sum_{\pi \in \Theta_{2k}} \prod_{l=1}^k \delta_{i_{\pi(2l-1)}, i_{\pi(2l)}}  |t_{\pi(2l)} - t_{\pi(2l-1)}|^{2H-2}.
	\]
	Let $\pi = \sigma \circ \tau^{-1}$. To check symmetry, apply a permutation $\rho \in \Theta_{2k}$ to the variables, mapping $t_i \to t_{\rho(i)}$ and we have:
	\[
	g(t_{\rho(1)}, \dots, t_{\rho(2k)}) = \sum_{\pi \in \Theta_{2k}} \prod_{l=1}^k \delta_{i_{\pi(2l-1)}, i_{\pi(2l)}}  |t_{\rho(\pi(2l))} - t_{\rho(\pi(2l-1))}|^{2H-2}.
	\]
	Since $|t_{\rho(\pi(2l))} - t_{\rho(\pi(2l-1))}| = |t_{\pi'(2l)} - t_{\pi'(2l-1)}|$ with $\pi' = \rho \circ \pi$, and the Kronecker deltas remain unchanged under reindexing, the sum is invariant:
	\[
	g(t_{\rho(1)}, \dots, t_{\rho(2k)}) = g(t_1, \dots, t_{2k}).
	\]
	Thus, $g$ is symmetric under all permutations of $\{t_1, \dots, t_{2k}\}$.
	Therefore, applying Remark \ref{SimplexCubeSymmetricFunc} we get:
	
	\[
	I = \frac{1}{(2k)!} \sum_{\pi \in \Theta_{2k}} \int_{[0,1]^{2k}} \prod_{l=1}^k \delta_{i_{\pi(2l-1)}, i_{\pi(2l)}}  |t_{\pi(2l)} - t_{\pi(2l-1)}|^{2H-2} \, dt_1 \dots dt_{2k}.
	\]
	For each $\pi$, if $\delta_{i_{\pi(2l-1)}, i_{\pi(2l)}} = 1$, compute:
	\[
	\int_{[0,1]^{2k}} \prod_{l=1}^k |t_{\pi(2l)} - t_{\pi(2l-1)}|^{2H-2} \, dt_1 \dots dt_{2k} = \prod_{l=1}^k \int_0^1 \int_0^1 |v_l - u_l|^{2H-2} \, du_l dv_l.
	\]
	Evaluate the pair integral:
	\[
	\int_0^1 \int_0^1 |v - u|^{2H-2} \, du dv = 2 \int_0^1 \int_0^u (u - v)^{2H-2} \, dv du = 2 \int_0^1 \frac{u^{2H-1}}{2H-1} \, du = \frac{2}{(2H-1)(2H)}.
	\]
	Thus:
	\[
	\int_{[0,1]^{2k}} \prod_{l=1}^k |t_{\pi(2l)} - t_{\pi(2l-1)}|^{2H-2} \, dt_1 \dots dt_{2k} = \left( \frac{2}{(2H-1)(2H)} \right)^k,
	\]
	assuming the maximum number of permutations is  $(2k)!$, we get:
	\[
	I \leq (2k)! \frac{1}{(2k)!} \left( \frac{2}{(2H-1)(2H)} \right)^k \leq  \left( \frac{2}{(2H-1)(2H)} \right)^k,
	\]
on the other hand we know $|{I_p \shuffle J_q}| = \binom{p+q}{p}.$
	Thus, the total integral over both simplices satisfies:
	
	\begin{equation}
		\sum_{\tau \in {I_p \shuffle J_q}} \sum_{\sigma \in \Theta_{2k}} I \leq  \frac{(2k)!(2k)!}{(2k)! p! q!} = \frac{(2k)!}{ p! q!} \left( \frac{2}{(2H-1)(2H)} \right)^k,
	\end{equation}
substituting this bound into the expectation formula:
	
	\[
	E \left[ \int_{\Delta^{p}[0,1]} dB^{I_p} \int_{\Delta^{q}[0,1]} dB^{J_q} \right] \leq  \frac{(2k)!}{ p! q!} \left( \frac{2}{(2H-1)(2H)} \right)^k,
	\]
	finally applying below scaling property of fractional Brownian motion, 
	\[
	\mathbb{E} \bigg[ \int_{\Delta^k[0,t]} dB^{I_k}\bigg]
	= t^{Hk} \mathbb{E} \bigg[ \int_{\Delta^k[0,1]} dB^{I_k} \bigg],
	\]
	we get
	\begin{align*}
		E \left[ \int_{\Delta^{p}[s,t]} dB^{I_p} \int_{\Delta^{q}[s,t]} dB^{J_q} \right] &\leq \frac{(2k)!}{k! 2^k p! q!} \left( \frac{2}{(2H-1)(2H)} \right)^k \times (t-s)^{2kH} \\
		&\le  \frac{\binom{2k}{p}}{k!2^k} 2^k (t-s)^{2kH}  \\
		&\le \frac{2^{2k}}{k!} (t-s)^{2kH} .
	\end{align*}
	
\end{proof}

\subsubsection*{Proof of Theorem \ref{SignatureExpectation}}

We first need to bound the kernel of fBms using the following lemma. The argument presented here is closely related to the proof of Proposition 4.1 in \cite{nualart2011construction}, with necessary adaptations to fit our framework.

\begin{lemma}\label{KernelBound1}
	Consider	below kernel $K(t, u)$ is given by
	
	\begin{equation}\label{KernelFormula}
	K(t, u) = c_H \left[ \left( \frac{u}{t} \right)^{1/2 - H} (t - u)^{H - 1/2} + \left( \frac{1}{2} - H \right) u^{1/2 - H} \int_u^t v^{H - 3/2} (v - u)^{H - 1/2} \, dv \right] 1_{\{0 < u < t\}},
	\end{equation}
	with $H \in (0, 1/2)$, $0 < u < t \leq T$. Then we have 
	\[
	|K(t, u)| \leq c_H \left[ (t - u)^{H - \frac{1}{2}} + u^{H - \frac{1}{2}} \right],
	\]
	with $ 0 < H < \frac{1}{2} $. 
\end{lemma}
\begin{proof}
	We want to find the constant $C$ in the inequality	
	\[
	|K(t, u)| \leq C \left[ (t - u)^{H - 1/2} + u^{H - 1/2} \right],
	\]
	Kernel in Equation (\ref{KernelFormula}) consists of two terms. For the first term 
	since $0 < u < t$, we have $\frac{u}{t} \leq 1$. Additionally, $1/2 - H > 0$ (because $H < 1/2$), so $\left( \frac{u}{t} \right)^{1/2 - H} \leq 1.$
	Thus $\left( \frac{u}{t} \right)^{1/2 - H} (t - u)^{H - 1/2} \leq (t - u)^{H - 1/2}.$
	This matches the first term in the desired bound.  Multiplying by $c_H$, the contribution of the first term to $|K(t, u)|$ is $	c_H (t - u)^{H - 1/2}$,
	suggesting a coefficient of $c_H$ for the $(t - u)^{H - 1/2}$ term in the bound. 
	Now we focus on the second term. 
	We want to find $C$ such that:
	
	\[
	\int_u^t v^{H - \frac{3}{2}} (v - u)^{H - \frac{1}{2}} \, dv  \le C u^{2H - 1}
	\]
	
	where $ 0 < u < t $ and $0<H<1/2$. Substitute $ v = u(w + 1) $, so $ v - u = u w $, $ dv = u \, dw $, with limits from $ w = 0 $ to $ w = \frac{t - u}{u} $. The integral becomes:
	
	\begin{align*}
		\int_u^t v^{H - \frac{3}{2}} (v - u)^{H - \frac{1}{2}} \, dv &= \int_u^t [u(w + 1)]^{H - \frac{3}{2}} (u w)^{H - \frac{1}{2}} (u \, dw) \\
		&= \int_u^t u^{H - \frac{3}{2}} (w + 1)^{H - \frac{3}{2}} u^{H - \frac{1}{2}} w^{H - \frac{1}{2}} u \, dw \\
		& = u^{2H - 1}  \int_u^t (w + 1)^{H - \frac{3}{2}} w^{H - \frac{1}{2}} \, dw
	\end{align*}
	
	In this step we need to find an upper bound for the below integral
	\[
\int_0^{\frac{t - u}{u}} (w + 1)^{H - \frac{3}{2}} w^{H - \frac{1}{2}} \, dw,
	\]
	Substitute $ z = \frac{w}{w + 1} $, so $ w = \frac{z}{1 - z} $, $ w + 1 = \frac{1}{1 - z} $, $ dw = \frac{dz}{(1 - z)^2} $, and limits from $ z = 0 $ to $ z = \frac{t - u}{t} $:
	\begin{align*}
		\int_0^{\frac{t-u}{u}} (w + 1)^{H - \frac{3}{2}} w^{H - \frac{1}{2}} \, dw &= \int_0^{\frac{t-u}{u}} (1 - z)^{\frac{3}{2} - H} \left( \frac{z}{1 - z} \right)^{H - \frac{1}{2}} \frac{dz}{(1 - z)^2} \\
		& = \int_0^{\frac{t - u}{t}} z^{H - \frac{1}{2}} (1 - z)^{-2H} \, dz = I, \\
		& \le B(H + \frac{1}{2}, 1 - 2H),
	\end{align*}
	since $ \frac{t - u}{t} < 1 $. Therefore, the coefficient for the second term is $(\frac{1}{2}-H) B(H + \frac{1}{2}, 1 - 2H) $ which is less than $1$, and results in $C = c_H$.

\end{proof}

We begin by deriving an upper bound for $\mathbb{E}[(D^{j}_{st})^2]$, followed by an explicit estimate for $\mathbb{E}[(C^{j}_{st})^2]$ through several steps. To obtain the bound for $\mathbb{E}[(C^{j}_{st})^2]$, we first establish an upper bound for $\mathbb{E}(J_{0s}(\nu))$.

\textbf{Bound for $\mathbb{E}[(D^{j}_{st})^2] $ }

\begin{lemma}
Consider
\[
D_{st}^j = \int_{s < u_{j+1} < \cdots < u_n < t} K_t^{\otimes (n-j)} \, dW,
\]
where  $ dW = dW_{u_{j+1}} \cdots dW_{u_n} $, a Stratonovich integral over the simplex $ s < u_{j+1} < \cdots < u_n < t $. 
indicating $ K_t^{\otimes (n-j)} $ involves $ n-j $ terms. 

\[
	\mathbb{E}\left[(D_{st}^j)^2\right] \leq \frac{2^{2n-2j}}{(n-j)!} \times (\frac{2c_H^2}{H})^{n-j}  (s-t)^{2H(n-j)}.
\]

\end{lemma}

\begin{proof}

Our goal is to find a proper upper bound for $\mathbb{E}\left[(D_{st}^j)^2\right]$.
For this purpose we need Remark \ref{SimplexCubeSymmetricFunc}.

Applying Remark \ref{SimplexCubeSymmetricFunc} and applying a modification of  Lemma 2.7 in \cite{nualart2011construction}, we get

{\begin{equation}\label{UpperBoundD}
			\mathbb{E}\left[(D_{st}^j)^2\right] \leq \frac{c_H^{2n-2j}}{(n-j)!}
			\left(\sum_{k=\lfloor (n-j)/2\rfloor}^{n-j}  \frac{\lvert D_n^k\rvert}{2^{\,n-k}}\right)^2
			\left(
			\int_{s}^t \left[ (t - u)^{H - \frac{1}{2}} + u^{H - \frac{1}{2}} \right]  du 
			\right)^{n-j}.
		\end{equation}
}

Now we need to bound $\int_v^t \left[(t-u)^{H - 1/2} + u^{H - 1/2} \right]^2 du$
where we have

\begin{align*}
	\int_s^t \left[(t-u)^{H - 1/2} + u^{H - 1/2} \right]^2 du 
	&\le 2\int_s^t (t-u)^{2H - 1} \,du + 2\int_s^t u^{2H - 1} \,du \\
	&= \frac{2}{2H} \left((t-s)^{2H} + (t^{2H} - s^{2H})\right). \\
	&\le \frac{2}{H} (t-s)^{2H},
\end{align*}
where the last inequality results from the fact that $ a^{\alpha} - b^{\alpha} \leq (a - b)^{\alpha} $ for any $ 0 \leq b < a $ and $ \alpha \in (0,1) $. Moreover, for the summation in Equation (\ref{UpperBoundD}) we have

\begin{align*}
	\bigg( \sum_{k=\lfloor n/2 \rfloor}^{n} \frac{\binom{k}{n-k}}{2^{n-k}} \bigg)^2 =  \bigg( \sum_{m=0}^{\lfloor n/2 \rfloor} \binom{n-m}{m} \frac{1}{2^{m}}\bigg)^2 \le 2^{2n}.
\end{align*}

\end{proof}

\begin{remark}
	Applying same argument for $Q_{st}^1$ we get
	
	\begin{equation*}\label{Q1Bound}
		\mathbb{E}\left[(Q_{st}^1)^2\right] \leq \frac{2^{2n}}{n!} \times (\frac{2c_H^2}{H})^{n} (t - s)^{2nH }    
	\end{equation*}

\end{remark}

\textbf{Bound on $\mathbb{E}\left[ \left(C_{st}^j\right)^2 \right]$}

Now we focus on finding constant in the below inequality which can be obtained in several steps. 

\begin{equation}
	\mathbb{E}\left[ \left(C_{st}^j\right)^2 \right] \le C_c (t - s)^{2jH }
\end{equation}

Recall that 
\begin{equation}\label{BoundJ}
	C_{st}^{j} = \sum_{k=\lfloor j/2 \rfloor}^{j} \frac{1}{2^{j-k}} \sum_{\nu \in D_j^k} J_{0s}(\nu),
\end{equation}
where
\[
J_{0s}(\nu) = \int_{0 < u_1 < \cdots < u_k < s} \partial Z_{u_1}^1 \cdots \partial Z_{u_k}^k,
\]
The summation is depends on what value $\nu$ take so we need to consider all possible cases for $\nu$. 
Therefore we need to obtain a bound for $J_{0s}(\nu)$ which works for all $\nu$. 
\begin{proposition}
	For $J_{0s}(\nu)$ in Equation (\ref{BoundJ}) we have 
	\[
	\mathbb{E}\left[J_{0s}(\nu)^2\right] \leq C(t - s)^{2kH}.
	\]
	where
\[
C = 	\frac{k \beta_{{k,H}}}{k!} \big( \frac{2}{H}\big)^{2k-2} ,
\]	

and $
	\beta_{{k,H}}= \frac{\pi (\frac{1}{2}-H)}{\cos(\pi H)} + \frac{1}{1-2kH} $.

\end{proposition}

In order to find constant $C_c$ we need some auxiliary lemmas. First we review Lemma 2.5 from \cite{nualart2011construction}. 
this lemma states that if $2kH < 1$ for $A > 0$ define 
	\begin{equation}\label{IntA}
	\beta_A = \int_0^A \left[ y^{H - \frac{1}{2}} - (1 + y)^{H - \frac{1}{2}} \right] \left[ y^{H - \frac{1}{2}} + (A - y)^{H - \frac{1}{2}} \right] y^{2(k - 1)H} \, dy.
\end{equation}
Then  $\displaystyle \sup_{A > 0} \beta_A < \infty$.
Furthermore, we can write $\sup_A \beta_A \leq \alpha + \gamma$, with
\[
\alpha = \int_0^\infty \left[ y^{H - \frac{1}{2}} - (1 + y)^{H - \frac{1}{2}} \right] y^{2(k - 1)H + H - \frac{1}{2}} \, dy,
\]
\[
\gamma = \int_0^A \left[ y^{H - \frac{1}{2}} - (1 + y)^{H - \frac{1}{2}} \right] (A - y)^{H - \frac{1}{2}} y^{2(k - 1)H} \, dy.
\]

\begin{lemma}\label{Alpha_int}
	Consider the integral:
	\[
	\alpha = \int_0^\infty \left[y^{H - \frac{1}{2}} - (1 + y)^{H - \frac{1}{2}}\right] y^{2(k - 1)H + H - \frac{1}{2}} \, dy,
	\]
	where $ 0 < H < \frac{1}{2} $ and $ 2kH < 1 $, then
	\[
	\alpha \leq \frac{1}{1-2kH}.
	\] 
\end{lemma}

\begin{proof}
	To estimate an upper bound, we split the integral into two parts:
	\[
	\alpha = \int_0^1 \left[y^{H - \frac{1}{2}} - (1 + y)^{H - \frac{1}{2}}\right] y^{2(k - 1)H + H - \frac{1}{2}} \, dy + \int_1^\infty \left[y^{H - \frac{1}{2}} - (1 + y)^{H - \frac{1}{2}}\right] y^{2(k - 1)H + H - \frac{1}{2}} \, dy.
	\]
	For small $ y $, we use the fact that:
	$(1 + y)^{H - 1/2} \geq 1,$
	and since $ H - \frac{1}{2} < 0 $, we have:
	$y^{H - \frac{1}{2}} - (1 + y)^{H - \frac{1}{2}} \leq y^{H - \frac{1}{2}}.$
	Thus,
	$\left|y^{H - \frac{1}{2}} - (1 + y)^{H - \frac{1}{2}}\right| y^{2(k - 1)H + H - \frac{1}{2}} \leq y^{2kH - 1},$
	and so:
	\[
	\int_0^1 \left[y^{H - \frac{1}{2}} - (1 + y)^{H - \frac{1}{2}}\right] y^{2(k - 1)H + H - \frac{1}{2}} \, dy \leq \int_0^1 y^{2kH - 1} dy = \frac{1}{2kH}.
	\]
	Now we need to consider integral behavior for $y\ge 1$. We consider the expression:
	$f(y) = y^{H - \frac{1}{2}} - (1 + y)^{H - \frac{1}{2}}, \quad \text{for } y > 1 \text{ and } H \in \left(0, \tfrac{1}{2}\right).$ 
	Since $ H - \frac{1}{2} < 0 $, the function $ x \mapsto x^{H - \frac{1}{2}} $ is decreasing. Therefore, $ f(y) > 0 $.
	We can represent the difference as an integral using the Fundamental Theorem of Calculus:
	\[
	f(y) = y^{H - \frac{1}{2}} - (1 + y)^{H - \frac{1}{2}}  = \int_y^{y+1} (\frac{1}{2}-H ) x^{H - \frac{3}{2}} \, dx.
	\]
	Now, because $ x \geq y \geq 1 $ and $ H - \frac{3}{2} < -1 $, the function $ x^{H - \frac{3}{2}} $ is positive and decreasing. Thus, for all $ x \in [y, y+1] $, we have $x^{H - \frac{3}{2}} \leq y^{H - \frac{3}{2}}$ and therefore,
	\[
	\int_y^{y+1} x^{H - \frac{3}{2}} \, dx \leq y^{H - \frac{3}{2}}.
	\]
	Putting it all together, we obtain $f(y) \leq (\frac{1}{2} - H) y^{H - \frac{3}{2}}.$ Therefore:
	\[
	\int_1^\infty \left[y^{H - \frac{1}{2}} - (1 + y)^{H - \frac{1}{2}}\right] y^{2(k - 1)H + H - \frac{1}{2}} \, dy \leq (1/2 - H ) \int_1^\infty y^{2kH - 2} dy = \frac{1/2 - H }{1 - 2kH},
	\]
	Combining both regions, we obtain:
	\begin{equation}\label{AlphaOLD}
		\alpha \leq \frac{1}{2kH} + \frac{\frac{1}{2} -H }{1 - 2kH},
	\end{equation}
	when  $ 0 < H < \frac{1}{2} $ and $ 2kH < 1 $.
	Finally we get $\alpha \le \frac{1}{1-2kH}$.
\end{proof}

In addition from \cite{nualart2011construction} we have
\begin{equation*}
	\gamma \leq  \left( \tfrac{1}{2} - H \right)^{2kH} \Gamma\left( \tfrac{1}{2} - H \right) \Gamma\left( H + \tfrac{1}{2} \right). \notag
\end{equation*}
Where $2kH\le 1$ we get that $\left( \tfrac{1}{2} - H \right)^{2kH} \le 1$. Moreover from reflection formula of the $\Gamma$ function we know that $\Gamma(z)\Gamma(1 - z) = \frac{\pi}{\sin(\pi z)}$ Therefore:
\[
\gamma \le \frac{\pi}{\cos(\pi H)}.
\]
Note that $\gamma$ is finite. Since the Gamma function argument is between 0 and 1, it diverges only in the case $H \to 1/2$ Gamma function argument tends to $1/2$. But when $H \to 1/2$ then $2kH \to 1$ therefore we have 
\begin{equation*}
	\gamma \leq  \left( \tfrac{1}{2} - H \right) \Gamma\left( \tfrac{1}{2} - H \right) \Gamma\left( H + \tfrac{1}{2} \right) = \Gamma\left( \tfrac{3}{2} -H \right) \Gamma\left( H + \tfrac{1}{2} \right),
\end{equation*}
Therefore the integral in Equation (\ref{IntA}) will be bounded by below constant
\begin{equation}\label{BoundIntA}
\beta_{{k,H}}= \frac{\pi (\frac{1}{2}-H)}{\cos(\pi H)} + \frac{1}{1-2kH}.
\end{equation}

\begin{lemma}\label{BoundG}
	Let
	\[
	G(u_b) = \int_{0 < u_b < u_{b+1} < \cdots < u_k < s} 
	\prod_{h = b+1}^{k} K(\tau_{m(h)-1}, u_h) K(\tau_{m(h)}, u_h) \, du_{b+1} \cdots du_k,
	\]
	then 
	\[
	|G(u_b)| \leq \frac{(2c_H^2)^{k-b}}{(k-b)!H^{k-b}} (s - u_b)^{2(k - b)H}.
	\]
	
\end{lemma}

\begin{proof}
	From Lemma \ref{KernelBound1} and since $s\le\tau_{m(h)}\le t$ we have
%
	\begin{align}
		|G(u_b)|  & \le c_H^{2(k-b)}\int_{0<u_b < u_{b+1}  < \cdots < u_k < s} \prod_{h=b+1}^{k} \left[(s - u_h)^{H - 1/2} + u_h^{H - 1/2} \right]^2 \, du_{b+1} \cdots du_k \notag \\
		&= c_H^{(k-b)} \int_{0<u_b < u_{b+1} < \cdots < u_k < s} \prod_{h=b+1}^{k} \left[(s - u_h)^{H - 1/2} + u_h^{H - 1/2} \right]^2 \, du_{b+1} \cdots du_k \notag \\
		&\leq  c_H^{(k-b)} \frac{2^{k-b}}{(k-b)!} \int_{[u_b, s]^{k-b}} \prod_{h=b+1}^{k} \left[(s - u_h)^{2H - 1} + u_h^{2H - 1} \right] \, du_{b+1} \cdots du_k \notag \\
		& = c_H^{(k-b)} \frac{2^{k-b}}{(k-b)!} \bigg( \int_{[u_b, s]} \left[(s - u_h)^{2H - 1} + u_h^{2H - 1} \right] \, du_{h} \bigg)^{k-b} \notag \\
		& \le  \frac{c_H^{2(k-b)}  2^{k-b}}{H^{k-b}(k-b)!} \bigg( (s-u_b)^{2H} \bigg)^{k-b} \notag
	\end{align}

\end{proof}

\begin{lemma}\label{MuEven}
	Assume term $J_{0s}(\nu)$ for for $\nu = (j_1, \ldots, j_k)$ = (2,2,\ldots,2) is defined  as follows
	\[
	J_{0s}(\nu) = \int_{0 < u_1 < \cdots < u_k < s} [K(t, u_1) - K(s, u_1)] K(\tau_2, u_1)
	\prod_{h=2}^k K(\tau_{2h-1}, u_h) K(\tau_{2h}, u_h) \, du_1 \cdots du_k.
	\]
	Then we have:
	\[
	\mathbb{E}|J_{0s}(\nu)| \leq C_{\nu_{2k}} (t - s)^{2kH},
	\]
	where 	
	\[
	C_{\nu_{2k}} = \frac{c_H^2}{(k-1)!} (\frac{2c_H^2}{H})^{k-1} \times \beta_{{k,H}}.
	\]
	
\end{lemma}

\begin{proof}
	As a consequence, of Lemma \ref{KernelBound1}, we have
	
	\begin{equation}
		|J_{0s}(\nu)| \leq c_H^{2k} \int_{0 < u_1 < \cdots < u_k < s} 
		\varphi_{u_1}^{(1)} \prod_{h=2}^k \varphi_{u_h}^{(2)} \, du_1 \cdots du_k,
	\end{equation}
	where
	\[
	\varphi_{u_1}^{(1)} = \left[(s - u_1)^{H - 1/2} - (t - u_1)^{H - 1/2} \right] 
	\left[(s - u_1)^{H - 1/2} + u_1^{H - 1/2} \right],
	\]
	and
	\[
	\varphi_{u_h}^{(2)} = \left[(s - u_h)^{H - 1/2} + u_h^{H - 1/2} \right]^2.
	\]
	Applying Lemma \ref{BoundG} for $b=1$ we get 
	$\int_{u_1 < u_2 < \cdots < u_k < s} \prod_{h=2}^{k} \varphi^{(2)}_{u_h} \, du_2 \cdots du_k \le \frac{1}{(k-1)!} (\frac{2}{H})^{k-1} (s - u_1)^{2(k - 1)H}
	$. Therefore, by making the change of variables $s - u_1 = v$ and $y = \frac{v}{t - s}$, we get
	
	\begin{align}
		|J_{0s}(\nu)| \le \, & \frac{1}{(k-1)!} (\frac{2c_H^2}{H})^{k-1}(t - s)^{2kH} \int_0^{s / (t - s)} 
		\left[ y^{H - \frac{1}{2}} - (1 + y)^{H - \frac{1}{2}} \right] \notag \\
		&\quad \times \left[ y^{H - \frac{1}{2}} + \left( \frac{s}{t - s} - y \right)^{H - \frac{1}{2}} \right]
		y^{2(k - 1)H} \, dy. \notag
	\end{align}
	
\end{proof}

\begin{remark}
	suppose that $\nu = (j_1, \dots ,j_{k-1}, 1)$. In this case 
	\[
	J_{0s}(\nu) = \int_{0 < u_1 < \cdots < u_k < s} \partial Z_{u_1}^1 \cdots \partial Z_{u_{k-1}}^{k - 1} K(\tau_j, u_k) \partial W_u(i_j),
	\]
	then from Ito isometry and Lemma \ref{KernelBound1} we get
	\begin{align*} 
		\mathbb{E}[J_{0s}(\nu)^2] &= \int_0^s \mathbb{E}(J_{0u}(\nu')^2) K(\tau_j, u)^2 \, du \\
		&\leq c_H^2 \int_0^s \mathbb{E}(J_{0u}(\nu')^2) \left( (s - u)^{2H - 1} + u^{2H - 1} \right) \, du,
	\end{align*}
	with $\nu'=(j_1,j_2,\ldots,j_{k-1})$.
\end{remark}

Now we consider the next possible case.

\begin{lemma}
	
Now we consider the case where $\nu = (j_1,\dots,j_{b-1},j_b, j_{b+1}, \dots,j_k)$ with $2 \le b$. Set $j_b = 1$ and $j_{b+1} = j_{b+2} = \dots = j_k = 2$. Then we have
\begin{align*}
	\mathbb{E}[J_{0s}(\nu)^2] 
\leq c_H^{2k-2b} \int_0^s \mathbb{E}[J_{0u_b}(\nu'')^2] \left[(s - u_b)^{2H - 1} + u_b^{2H - 1}\right](s - u_b)^{4(k - b)H} \, du_b.
\end{align*}
	
\end{lemma}

\begin{proof}
In this case applying Fubini's theorem 'we have 
\[
J_{0s}(\nu) = \int_0^s J_{0u_b}(\nu') K(\tau_{m(h)}, u_b) G(u_b) \, dW_{u_b}^{i_{m(h)}},
\]

with $\nu' = (j_1, \ldots, j_{b-1})$, and where
\[
G(u_b) = \int_{0 < u_b < u_{b+1} < \cdots < u_k < s} \prod_{h=b+1}^k K(\tau_{m(h)-1}, u_h) K(\tau^{m(h)}, u_h) \, du_{b+1} \cdots du_k.
\]

From Lemma \ref{BoundG} for $G(u_b)$ and Ito isometry we get 
\begin{align*}
	\mathbb{E}[J_{0s}(\nu)^2] 
	&= \int_0^s \mathbb{E}[J_{0u_b}(\nu')^2] \, K(\tau_{m(h)}, u_b)^2 \, G(u_b)^2 \, du_b \\
	&\leq c_H^{2k-2b} \int_0^s \mathbb{E}[J_{0u_b}(\nu')^2] \left[(s - u_b)^{2H - 1} + u_b^{2H - 1}\right](s - u_b)^{4(k - b)H} \, du_b.
\end{align*}

\end{proof}

\begin{lemma}
	If $\varphi^{(2)}_{u_h} = \left[ (s - u_h)^{H - 1/2} + u_h^{H - 1/2} \right]^2,$ then    
	\[
	\int_{u_1 < \cdots < u_k < s} \prod_{h=2}^{k} \varphi^{(2)}_{u_h} \, du_2 \cdots du_k \leq \frac{2^{k-1}}{(k-1)!} \int_{[u_1, s]^{k-1}} \prod_{h=2}^{k} \left[ (s - u_h)^{2H - 1} + (u_h - u_1)^{2H - 1} \right] \, du_2 \cdots du_k.
	\]
	
\end{lemma}

\begin{proof}
	
	Using $ (a+b)^2 \leq 2a^2 + 2b^2 $, we have
	$
	\varphi^{(2)}_{u_h} = (s - u_h)^{2H - 1} + 2(s - u_h)^{H - 1/2} u_h^{H - 1/2} + u_h^{2H - 1} \leq 2(s - u_h)^{2H - 1} + 2u_h^{2H - 1}
	$
	Moreover, since $ u_h > u_1 $ and $ 2H - 1 < 0 $ (as $ 0 < H < \frac{1}{2} $), then $u_h^{2H - 1} \leq (u_h - u_1)^{2H - 1}$.
	Therefore:
	\[
	\prod_{h=2}^k \varphi^{(2)}_{u_h} \leq 2^{k-1} \prod_{h=2}^k \left[ (s - u_h)^{2H - 1} + (u_h - u_1)^{2H - 1} \right],
	\]
	
	Note that the ordered integral $ {u_1 < \cdots < u_k < s} $ is bounded by the larger domain $ {[u_1, s]^{k-1}} $.
	We also need to relate the integral over the ordered simplex to the integral over the cube. 
	Now we need to apply Remark \ref{SimplexCubeSymmetricFunc}. In our case, the integrand on the right,
	\[
	g(u_2, \ldots, u_k) = \prod_{h=2}^{k} \left[ (s - u_h)^{2H - 1} + (u_h - u_1)^{2H - 1} \right],
	\]
	is symmetric because permuting the variables $ u_2, \ldots, u_k $ only reorders the factors in the product, leaving the product unchanged. 
	Rearranging, we obtain:
	
	\[
	\int_{u_1 < \cdots < u_k < s} g(u_2, \ldots, u_k) \, du_2 \cdots du_k = \frac{1}{(k-1)!} \int_{[u_1, s]^{k-1}} g(u_2, \ldots, u_k) \, du_2 \cdots du_k.
	\]
	This justifies dividing by $ \frac{1}{(k-1)!} $, as the symmetry ensures that each of the $(k-1)!$ ordered regions contributes equally to the cube integral.
\end{proof}

\begin{lemma}\label{LemmaExpectationJDoubleintBound}
	Let $\nu = (j_1,\ldots,j_c,j_{c+1},\ldots,j_{b-1},j_b,j_{b+1},\ldots,j_k)$ the sequence of indices $ j_k, j_{k-1}, \ldots, j_{b+1} = 2 $, $ j_b = 1 $, $ j_{b-1}, \ldots, j_{c+1} = 2 $, and $ j_c = 1 $, where $ 2 \leq c \leq b $, and $ \mathbb{E}[J_{0s}(\nu)^2] $ incorporated by two integration stages corresponding to the transitions at $ j_b = 1 $ and $ j_c = 1 $. Then	
\begin{equation}
		\mathbb{E}\left[J_{0s}(\nu)^2\right] 
		\leq C_{\text{single}} \int_0^s \mathbb{E}\left[J_{0u_c}(\nu')^2\right] 
		\left[ (s - u_c)^{2H - 1} + u_c^{2H - 1} \right]  \times (s - u_c)^{4(k - c)H + 2H} \, du_c,
	\end{equation}
	where 
	\[
	C_{\text{single}} =2 \frac{2^{k - c}}{(k - c)!}\frac{(2c_H)^{2k-2c}}{H^{2k-2c}}.
	\]
	
\end{lemma}

\begin{proof}
We perform the proof in two steps. In the first step we show that 
	\begin{align}\label{ExpectationJDoubleintBound}
		\mathbb{E}[J_{0s}(\nu)^2] 
		&\leq \frac{2^{k - c}}{(k - c)!} 4c_H^4 \frac{(2c_H)^{2k-2c}}{H^{2k-2c}}  \int_{0 < u_c < u_b < s} 
		\mathbb{E}[J_{0u_c}(\nu'')^2] 
		\left[(u_b - u_c)^{2H - 1} + u_c^{2H - 1} \right] \notag \\
		&\quad \times (u_b - u_c)^{4(b - c)H} 
		\left[(s - u_b)^{2H - 1} + u_b^{2H - 1} \right] \notag \\
		&\quad \times (s - u_b)^{4(k - b)H} \, du_c \, du_b,
	\end{align}

where $\nu'' = (j_1,\ldots,j_{c-1})$.

	The goal is to compute $ \mathbb{E}[J_{0s}(\nu)^2] $ by incorporating two integration stages corresponding to the transitions at $ j_b = 1 $ and $ j_c = 1 $.
	Applying Ito isometry and Fubini's theorem we get:
	\begin{align*}
		\mathbb{E}[J_{0s}(\nu)^2] &= \int_0^{s} \mathbb{E}[J_{0u_b}(\nu')^2] K(\tau_{m(h)}, u_b)^2 G(u_b)^2 \, du_b \\
		& \leq C_1 \int_0^s \mathbb{E}\left[J_{0u_b}(\nu')^2\right] 
		\left[(s - u_b)^{2H - 1} + u_b^{2H - 1}\right](s - u_b)^{4(k - b)H} \, du_b.
	\end{align*}
	where $\nu' = (j_1,\ldots,j_{b-1})$ and inequality results from Lemma \ref{KernelBound1} and \ref{BoundG} with constant
	\[
	C_1 = \frac{1}{(k-b)!} \frac{(2c_H)^{2k-2b}}{H^{2k-2b}}.
	\]
	with the same technique we get
	\[
	\mathbb{E}[J_{0u_b}(\nu')^2] \leq C_2 \int_0^{u_b} \mathbb{E}[J_{0u_c}(\nu'')^2] \left[(u_b - u_c)^{2H - 1} + u_c^{2H - 1}\right](u_b - u_c)^{4(b - c)H} \, du_c,
	\]
	where 
	\begin{equation}\label{simplification}
	C_2 = \frac{1}{(k-b)!}\frac{1}{(b-c)!} \frac{(2c_H)^{2k-2c}}{H^{2k-2c}} \le \frac{2^{k-c}}{(k-c)!} \frac{(2c_H)^{2k-2c}}{H^{2k-2c}},
	\end{equation}
because
	\[
	\frac{1}{(k - b)! (b - c)!} = \frac{1}{(k - c)!} \binom{k - c}{k - b} \le \frac{2^{k-c}}{(k-c)!}.
	\]

	Therefore, we obtain the desired inequality. 
%

In the next lemma we aim to transform double integration in Equation (\ref{ExpectationJDoubleintBound}) to one integration in the following lemma.

\begin{lemma}
	With condition in Lemma \ref{LemmaExpectationJDoubleintBound}
	below inequality holds:
	\begin{align}
		\mathbb{E}\left[J_{0s}(\nu)^2\right] 
		&\leq C_{\text{single}} \int_0^s \mathbb{E}\left[J_{0u_c}(\nu')^2\right] 
		\left[ (s - u_c)^{2H - 1} + u_c^{2H - 1} \right] \notag \\
		&\quad \times (s - u_c)^{4(k - c)H + 2H} \, du_c.
	\end{align}
		where 
	\[
	C_{\text{single}} =  \frac{2^{k - c}}{(k - c)!}\frac{(2c_H)^{2k-2c}}{H^{2k-2c}}.
	\]
\end{lemma}

	Integral in (\ref{ExpectationJDoubleintBound}) is a double integral over the region $ 0 < u_c < u_b < s $.
	\begin{align*}
		\mathbb{E}[J_{0s}(\nu)^2] &\leq \frac{2^{k-c}}{(k-c)!} \frac{(2c_H)^{2k-2c}}{H^{2k-2c}} \int_{0 < u_c < u_b < s} \mathbb{E}[J_{0u_c}(\nu')^2] \left[(u_b - u_c)^{2H - 1} + u_c^{2H - 1}\right] \\
		& \times (u_b - u_c)^{4(b - c)H} \left[(s - u_b)^{2H - 1} + u_b^{2H - 1}\right] (s - u_b)^{4(k - b)H} \, du_c \, du_b,
	\end{align*}
	Since $ u_b^{2H - 1} \le (u_b - u_c)^{2H - 1} $ for $0<H<1/2$ we can write:
	\begin{align*}
		\mathbb{E}[J_{0s}(\nu)^2] &\leq \frac{2^{k-c}}{(k-c)!} \frac{(2c_H)^{2k-2c}}{H^{2k-2c}} \int_{0 < u_c < u_b < s} \mathbb{E}[J_{0u_c}(\nu')^2] \left[ (u_b - u_c)^{2H - 1} + u_c^{2H - 1} \right] \\
		&\times (u_b - u_c)^{4(b - c)H} \left[ (s - u_b)^{2H - 1} + (u_b - u_c)^{2H - 1} \right] (s - u_b)^{4(k - b)H} \, du_c \, du_b.
	\end{align*}
	We rewrite the double integral as:
	\begin{align*}
		&\int_0^s \int_{u_c}^s \mathbb{E}[J_{0u_c}(\nu')^2] \left[ (u_b - u_c)^{2H - 1} + u_c^{2H - 1} \right] (u_b - u_c)^{4(b - c)H} 
		\\ 
		&\times \left[ (s - u_b)^{2H - 1} + (u_b - u_c)^{2H - 1} \right] (s - u_b)^{4(k - b)H} \, du_b \, du_c.
	\end{align*}
	The inner integral is:
	\begin{align*}
		I(u_c) & = \int_{u_c}^s \left[ (u_b - u_c)^{2H - 1} + u_c^{2H - 1} \right] (u_b - u_c)^{4(b - c)H} \\
		& \times \left[ (s - u_b)^{2H - 1} + (u_b - u_c)^{2H - 1} \right] (s - u_b)^{4(k - b)H} \, du_b.
	\end{align*}
	by expanding integrand 
	we yield four terms:
	\begin{align*}
		& \left[ (u_b - u_c)^{2H - 1} (s - u_b)^{2H - 1} + (u_b - u_c)^{4H - 2} + u_c^{2H - 1} (s - u_b)^{2H - 1} + u_c^{2H - 1} (u_b - u_c)^{2H - 1} \right] \\
		& \times (u_b - u_c)^{4(b - c)H} (s - u_b)^{4(k - b)H}.
	\end{align*}
	
	Substitute $ t = u_b - u_c $, so $ u_b = t + u_c $, $ du_b = dt $, $ s - u_b = (s - u_c) - t $, with limits $ t $ from 0 to $ s - u_c $. The integral becomes:
	\begin{align*}
		I(u_c) &= \int_0^{s - u_c}  t^{2H - 1 + 4(b - c)H} [(s - u_c) - t]^{2H - 1 + 4(k -b)H} dt + t^{4H - 2 + 4(b - c)H} [(s - u_c) - t]^{4(k - b)H} dt \\
		&+ u_c^{2H - 1} t^{4(b - c)H} [(s - u_c) - t]^{2H - 1 dt + 4(k - b)H} + u_c^{2H - 1} t^{2H - 1 + 4(b - c)H} [(s - u_c) - t]^{4(k - b)H}  dt.
	\end{align*}
	
	Use the Beta function, $ \int_0^a t^p (a - t)^q \, dt = a^{p + q + 1} B(p + 1, q + 1) $:

	\begin{itemize}
		\item First term: $ (s - u_c)^{4H - 1 + 4(k - c)H} B(2H + 4(b - c)H, 2H + 4(k - b)H) \\ = (s - u_c)^{4H - 1 + 4(k - c)H} B_1 $ ,
		
		\item Second term: $ (s - u_c)^{4H - 1 + 4(k - c)H} B(4H - 1 + 4(b - c)H, 4(k - b)H + 1) \\ = (s - u_c)^{4H - 1 + 4(k - c)H} B_2 $,
		
		\item Third term: $ u_c^{2H - 1} (s - u_c)^{2H + 4(k - c)H} B(4(b - c)H + 1, 2H + 4(k - b)H) \\ = u_c^{2H - 1} (s - u_c)^{2H + 4(k - c)H} B_3 $,
		
		\item Fourth term: $ u_c^{2H - 1} (s - u_c)^{2H + 4(k - c)H} B(2H + 4(b - c)H, 4(k - b)H + 1) \\ = u_c^{2H - 1} (s - u_c)^{2H + 4(k - c)H} B_4 $.

	\end{itemize}
	
	Therefore, since we can factorize $(s - u_c)^{2H-1}$ in the first and second term and $u_c^{2H - 1}$ in the third and fourth   term we have

	\[
	I(u_c)\leq C \int_0^s  \left[ (s - u_c)^{2H - 1} + u_c^{2H - 1} \right] (s - u_c)^{4(k - c)H + 2H} \, du_c,
	\]
	where $C = 2 \max \{B_1, B_2, B_3, B_4\} $.  Form \cite{ALZERSharpInequality} we know that $ B(x, y) \leq \frac{1}{xy} \le 1 $ for $x,y\ge 1$.
	If we assume $H$ is large enough then we get 
	\[
	\max \{ B_1, B_2, B_3, B_4 \} \leq 1
	\] 
If we assume every beta function argument is larger than 1. In this case we can simply use bound equal to "one". Therefore we have:
		\[
		C_{\text{single}}  =2 \frac{2^{k - c}}{(k - c)!} \frac{(2c_H)^{2k-2c}}{H^{2k-2c}}  
		\]

	\end{proof}
	
	Now consider below 
	\[
	\mathbb{E}[J_{0s}(\nu)^2] \leq C_{\text{single}}  \int_0^s \mathbb{E}[J_{0u_c}(\nu')^2] \left[ (s-u_c)^{2H-1} + u_c^{2H-1} \right] (s-u_c)^{4(k-c)H+2H} \, du_c.
	\]
	
	To get above inequality we need to perform at most  $(k-r-2)/2$ . By induction: 
	\[
	\mathbb{E}[J_{0s}(\nu)^2] \leq C_I \int_0^s \mathbb{E}[J_{0u}(\nu')^2] \left[(s-u)^{2H-1} + u^{2H-1} \right] (s-u)^{2H\sum_{l=r(\nu)+2}^k j_l} \, du.
	\]
	
	If we do not use simplification in Equation (\ref{simplification}), then   $C_I(\nu)$ will be obtained from the last step that
	\[
	C_I= (\frac{2c_H}{H})^{2k-2r}.
	\]

	
Now we can focus on the general $\nu$. By iteration we achieve 
		\begin{equation*}
	\mathbb{E}[J_{0s}(\nu)^2] \leq  C_I \int_0^s \mathbb{E}[J_{0u}(\nu')^2] \left[(s - u)^{2H - 1} + u^{2H - 1}\right] 
	\times (s - u)^{2H \sum_{\ell = r+2}^k j_\ell} \, du, 
\end{equation*}

where  $\nu'$ are obtain a multi-index of length $r$ and can be either $(1,2,\dots,2)$  or $(2,\dots,2)$ with $j_{r+1}=1$. Now we investigate both cases. Assume  $\nu'= (1,2,\dots,2)$. In this case we get 
\[
J_{0s}(\nu') = \int_0^u [K(t, u_1) - K(s, u_1)] F(u_1) \, dW_{u_1}^{i_1},
\]
where 
\[
F(u_1) = \int_{u_1 < u_2 < \cdots < u_r < u} \prod_{h=2}^r K(\tau_{m(h)-1}, u_h) K(\tau_{m(h)}, u_h) \, du_2 \cdots du_r.
\]
From Lemma \ref{BoundG}	we know that 
\[
|F(u_1)| \leq C_F (u - u_1)^{2(r - 1)H}.
\]
where $C_F = \frac{(2c_H)^{k-1}}{H^{k-1}}$.

	\begin{lemma}
		Assume 
		\begin{equation*}
			\mathbb{E}[J_{0s}(\nu)^2] \leq C_I \int_0^s \mathbb{E}[J_{0u}(\nu')^2] \left[(s - u)^{2H - 1} + u^{2H - 1}\right] 
			\times (s - u)^{2H \sum_{\ell = r+2}^k j_\ell} \, du, 
		\end{equation*}
 and 
		\begin{equation*}
			\mathbb{E}[J_{0s}(\nu')^2] \leq C_F \int_0^u \left[(t - u_1)^{H - 1/2} - (s - u_1)^{H - 1/2}\right]^2 (u - u_1)^{4(r - 1)H} \, du_1,
		\end{equation*}
then
		\[
		\mathbb{E}[|J_{0s}(v)|^2] \leq C_I C_F \int_0^s \left[ (t - u)^{H - 1/2} - (s - u)^{H - 1/2} \right]^2 (s - u)^{2(j-1)H} \, du.
		\]
	\end{lemma}

	\begin{proof}
		
		by changing the order of integration we have:		
		\[
		\int_0^s \left[ (t - u_1)^{H - 1/2} - (s - u_1)^{H - 1/2} \right]^2\int_{u_1}^s  (u - u_1)^{4(r-1)H} \left[ (s - u)^{2H-1} + u^{2H-1} \right] (s - u)^{2H \sum_{l=r+2}^{k} j_l} \, du \, du_1.
		\]
		We consider inner integral and try to simplify it
		\begin{align*}
			I(u_1) &= \int_{u_1}^s (u - u_1)^{4(r-1)H} \left[ (s - u)^{2H-1} + u^{2H-1} \right] (s - u)^{2H \sum_{l=r+2}^{k} j_l} \, du \\
			&\le  \int_{u_1}^s (u - u_1)^{4(r-1)H} \left[ (s - u)^{2H-1} + (u-u_1)^{2H-1} \right] (s - u)^{2H \sum_{l=r+2}^{k} j_l} \, du.
		\end{align*}
		By substituting $ t = u - u_1 $, $ u = t + u_1 $, $ du = dt $, we have
		
		\begin{align*}
			I(u_1) &= \int_0^{s - u_1} t^{4(r-1)H} \left[ (s - u_1 - t)^{2H-1} + t^{2H-1} \right] (s - u_1 - t)^{2H \sum_{l=r+2}^{k} j_l} \, dt \\
			&= \int_0^{s - u_1} t^{4(r-1)H} (s - u_1 - t)^{2H-1 + 2H \sum_{l=r+2}^{k} j_l} \, dt \\
			&+ \int_0^{s - u_1} t^{4(r-1)H} t^{2H-1} (s - u_1 - t)^{2H \sum_{l=r+2}^{k} j_l} \, dt.
		\end{align*}
		Using the fact that
		\[
		\int_0^a t^{p - 1}(a - t)^{q - 1} \, dt = a^{p + q - 1} B(p, q), \quad \text{for } p, q > 0,
		\]
		for the first term is equal to:
		\[
		(s - u_1)^{4(r-1)H + 2H \left(1 + \sum_{l=r+2}^{k} j_l\right)} B\left(4(r-1)H + 1, 2H + 2H \sum_{l=r+2}^{k} j_l\right).
		\]
		In the second term we have:
		\[
		(s - u_1)^{4(r-1)H + 2H \sum_{l=r+2}^{k} j_l} B\left(4(r-2)H , 2H \sum_{l=r+2}^{k} j_l + 1\right),
		\]
		and since $
		4(r-1)H + 2H \sum_{l=r+2}^{k} j_l + 2H = 2(j-1)H,
		$
		\[
		I(u_1) \leq (s - u_1)^{2(j-1)H} B_1 +  (s - u_1)^{2(j-1)H} B_2 \le \max(B_1,B_2) (s - u_1)^{2(j-1)H}.
		\]

		For simplicity, we assume  $2H(m+1)\ge 1$,  (This is true specially when we have $H>1/4$ and $m\ge 1$), and $(4r-2)H\ge 1$ therefore we have
		\[
		\max \{ B_1, B_2 \} \leq 1.
		\]
		
	\end{proof}

Therefore for the first case we get
	\[
	\begin{aligned}
		\mathbb{E}\left[J_{0s}(\nu)^2\right] &\leq C_I C_F (t-s)^{2jH} \int_0^{s/(t-s)} \left[ (1+y)^{H-1/2} - y^{H-1/2} \right]^2 y^{2(j-1)H} \, dy \\
		&\leq 	\alpha'  C_I C_F (t-s)^{2jH},
	\end{aligned}
	\]
and we just need to find an upper bound for $	\alpha'  = \int_0^{s/(t-s)} \left[ (1+y)^{H-1/2} - y^{H-1/2} \right]^2 y^{2(j-1)H} \, dy $. We do it in the next lemma.

	\begin{lemma}
	The below integral is bounded by
	\[
	\alpha' = \int_0^\infty \left[ y^{H - 1/2} - (1 + y)^{H - 1/2} \right]^2 y^{2(k - 1)H} \, dy \le \frac{1}{2kH} + \left( H - \frac{1}{2} \right)^2  \frac{1}{2 - 2kH},
	\]
	where we assume that $ 2kH < 1 $ and $ 0 < H < 1/2 $.
\end{lemma}
\begin{proof}
	
	Assume that
	\[
	f(y) = \left[ y^{H - 1/2} - (1 + y)^{H - 1/2} \right]^2 y^{2kH - 2H},
	\]
	
	first consider $ 0\le y \le 1 $. Similar to  Lemma \ref{Alpha_int} we have
	\[
	f(y) \le y^{2H - 1}  y^{2kH - 2H} = y^{2kH - 1},
	\]
	
	and for $1\le y$
	\[
	f(y) \le (H - 1/2)^2 y^{2H - 3}  y^{2kH - 2H} = (H - 1/2)^2 y^{2kH - 3}.
	\]
	
	For $ y \in [0, 1] $, $ (1 + y)^{H - 1/2} \leq 2^{H - 1/2} $, so:
	\[
	\int_0^1 g(y)^2 y^{2kH - 2H} \, dy \leq \int_0^1 y^{2H - 1}  y^{2kH - 2H} \, dy = \int_0^1 y^{2kH - 1} \, dy = \left. \frac{y^{2kH}}{2kH} \right|_0^1 = \frac{1}{2kH}.
	\]
	For $ y \geq 1 $, $ g(y) \le (H - 1/2) y^{H - 3/2} $, we have:
	\[
	g(y)^2 \leq (H - 1/2)^2 y^{2H - 3},
	\]
	and		
	\begin{align*}
		\int_1^\infty g(y)^2 y^{2kH - 2H} \, dy &\leq (H - 1/2)^2 \int_1^\infty y^{2H - 3}  y^{2kH - 2H} \, dy \\
		& = (H - 1/2)^2 \int_1^\infty y^{2kH - 3} \, dy \\
		&=  \left( H - \frac{1}{2} \right)^2  \frac{1}{2 - 2kH}
	\end{align*}

	Therefore after combining two bounds we get
	\[
	\alpha_A' \leq \frac{1}{2kH} + \left( H - \frac{1}{2} \right)^2  \frac{1}{2 - 2kH}.
	\]
	
\end{proof}

However due to Equation (\ref{AlphaOLD}) the value of $\alpha'$ is less than $\alpha$ and since for $\nu'= (2,\dots,2)$ and $c_H < 1$ we have the same procedure we achieve that:

\begin{equation}
	\begin{aligned}
	\mathbb{E}\left[J_{0s}(\nu_k)^2\right]  &\leq \alpha C_I C_F (t-s)^{2kH} \\
	&\le 	\beta_{{k,H}}  \frac{(2c_H)^{k-1}}{H^{k-1}}  \frac{(2c_H^2)^{k-1}}{(k-1)!H^{k-1}} (t-s)^{2kH} \\
	& \le 	\frac{k \beta_{{k,H}}}{k!} \big( \frac{2}{H}\big)^{2k-2} (t-s)^{2kH}
\end{aligned}
\end{equation}

	\begin{corollary}
		
		We have 
		\begin{align*}
			\mathbb{E}(B_{st}^j) \le C_B (t - s)^{2nH },
		\end{align*}
		
		where 
		\[
		C_B = \beta_{j,H} H^{-2j}\frac{2^{2n+2j}}{(n-j)!j!}.
		\]

	\end{corollary}

	\begin{proof}
		We know that 
				\begin{align*}
			\mathbb{E}[(B_{st}^j)^2]& \le \mathbb{E}[(C_{st}^j)^2] \mathbb{E}[(D_{st}^j)^2] \\
			 &\le \left(\frac{2^{2n-2j}}{(n-j)!} \times (\frac{2}{H})^{n-j} \right) (t - s)^{2(n-j)H }  \bigg( \sum_{k=\lfloor j/2 \rfloor}^{j} \frac{1}{2^{j-k}} \sum_{v \in D_j^k} \bigg)^2 \mathbb{E}\left[J_{0s}(\nu_k)^2\right]  \\
			&\le \frac{2^{2n}}{(n-j)!} \times (\frac{2}{H})^{n-j}  \frac{j \beta_{{j,H}}}{j!} \big( \frac{2}{H}\big)^{2j-2}  (t - s)^{2nH } \\
			& \le  \frac{2^{3n}  n}{H^{n}  } \times \frac{2^j \beta_{j,H}}{(n-j)! j! H^j}  (t - s)^{2nH } \\
			& \le  \frac{2^{5n}  n \beta_{n,H}}{H^{2n} n! } (t - s)^{2nH }.
		\end{align*}

	\end{proof}

	Recall that $Q_{st} = Q_{st}^2 + Q_{st}^1$  and 
	
	\begin{equation}\label{Eq:BoundQ1}
		\mathbb{E}\left[(Q_{st}^1)^2\right] \leq \frac{2^{2n}}{n!} \times (\frac{2c_H^2}{H})^{n} (t - s)^{2nH } \le \frac{2^{3n}}{n!H^n} (t - s)^{2nH }.
	\end{equation}
	
	If $Q_{st}^2 = \sum_{j=1}^n B_{st}^j$ then 
we get 
\[ 
\mathbb{E}[(Q_{st}^2)^2] \le  \frac{2^{5n}  n^3}{H^{2n} n! } (t - s)^{2nH },
\]
as a result 
\[ 
\mathbb{E}[(Q_{st})^2] \le 2 \frac{2^{5n}  3}{H^{2n} n! } (t - s)^{2nH }+  + 2 \frac{2^{3n}}{n!H^n} (t - s)^{2nH } =  \frac{2^{3n+1}}{n!H^n} (t - s)^{2nH } \times (1+\frac{2^{2n}n^3 \beta_{n,H}}{H^n}).
\]

	Finally we are in the stage that we can 
	find constant in $\mathbb{E}\left[(Q_{st})^2\right] \leq C_q (t - s)^{2nH } $.
	From above result and and Equation we have (\ref{Eq:BoundQ1})
	
	\[
	C_q = \frac{1}{n!} \frac{2^{3n}}{H^n} \times \left( 1 + n \beta_{n,H}  \left(1 + \frac{2}{H}\right)^n\right),
	\]
	
	where $\beta_{n,H} = \frac{\pi (\frac{1}{2}-H)}{\cos(\pi H)} + \frac{1}{1-2nH} $.
	
	In order to find the overall bound we need Equation (\ref{RoughPathTerms}).
Therefore we get
\begin{align*}
\mathbb{E}\left[\left|B^n\right|^2\right] &= \mathbb{E}\left[\left|\sum_{j=1}^{n-1} \hat{B}^{nj}\right|^2\right] \\
& \le (n-1)^2 2^{n-1} \mathbb{E}(Q_{st}^2) \\
& \le        \frac{2^{4n}}{n!H^n}   (n^2+\frac{2^{2n}n^5 \beta_{n,H}}{H^n}) \times (t - s)^{2nH }
\end{align*}
where the factor $2^{n-1}$, is resulted from the number of shuffle terms when index $j$ changes and $\beta_{n,H} = \frac{\pi (\frac{1}{2}-H)}{\cos(\pi H)} + \frac{1}{1-2nH} $. Note that for large $n$ this upper bound is decaying.

\subsubsection*{Upper bound for the first moment of signature}

In Lemma \ref{SignatureFirstMoment} we want to show that if $n = 2k$ then

	\[
\mathbb{E}[S_{I_n}(\boldsymbol{B}_{t})_{st}]  \le \frac{{\beta_{{k,H}}}}{k!H^k} (t - s)^{nH}
\]
	where
	\begin{equation*}\label{BoundIntA}
			{\beta_{{k,H}}}= \frac{\pi (\frac{1}{2}-H)}{\cos(\pi H)} + \frac{1}{1-2kH} 
		\end{equation*}

\begin{proof}
	
	Applying Proposition \ref{StratonovichIto} signature expectation will be nonzero unless its decomposition includes sequence of lenght $k = n/2$, with all elements equal to 2, i.e. $\nu = (2,\ldots,2)$. Using Lemma \ref{MuEven} and the fact that $c_H < 1$ we have:
	
	\[
	\mathbb{E}|J_{st}(\nu)| \leq \frac{\beta_{{i,H}}}{(i-1)!} (\frac{2}{H})^{i-1} \times  (t-s)^{2kH}
	\]
%

	Therefore applying Equation (\ref{StratonovichItoFormula}) and considering the fact that there is only one nonzero term   we get:
	\[
	\mathbb{E}[S_{I_n}(\boldsymbol{B}_{t})_{st}] = \sum_{i = \lfloor n/2 \rfloor}^{n} \frac{1}{2^{n-i}} 
	\sum_{\nu \in D_n^i} \mathbb{E} (J_{st}(\nu)) \le \frac{{\beta_{{k,H}}}}{k!H^k} (t - s)^{nH}
	\]

\end{proof}

\end{document}